\documentclass[11pt]{article}
\usepackage{amsmath,amscd,amsfonts,amsthm,amssymb,url}
\usepackage{longtable}
\usepackage[all]{xypic}

\usepackage[usenames,dvipsnames]{color}

\newcommand{\shrinkmargins}[1]{
  \addtolength{\textheight}{#1\topmargin}
  \addtolength{\textheight}{#1\topmargin}
  \addtolength{\textwidth}{#1\oddsidemargin}
  \addtolength{\textwidth}{#1\evensidemargin}
  \addtolength{\topmargin}{-#1\topmargin}
  \addtolength{\oddsidemargin}{-#1\oddsidemargin}
  \addtolength{\evensidemargin}{-#1\evensidemargin}
  }

\shrinkmargins{.7}

\DeclareMathOperator{\MG}{MG}
\DeclareMathOperator{\TMG}{TMG}
\DeclareMathOperator{\no}{no}
\DeclareMathOperator{\Aut}{Aut}
\DeclareMathOperator{\Hom}{Hom}
\DeclareMathOperator{\car}{char}
\DeclareMathOperator{\SL}{SL}

\DeclareMathOperator{\GL}{GL}
\DeclareMathOperator{\Gr}{Gr}
\DeclareMathOperator{\M}{M}
\DeclareMathOperator{\Sp}{Sp}
\DeclareMathOperator{\GSp}{GSp}
\DeclareMathOperator{\Cl}{Cl}

\DeclareMathOperator{\End}{End}

\DeclareMathOperator{\Spec}{Spec}

\DeclareMathOperator{\Frob}{Frob}

\DeclareMathOperator{\Jac}{Jac}

\DeclareMathOperator{\coker}{coker}
\DeclareMathOperator{\im}{im}
\DeclareMathOperator{\rank}{rank}
\newcommand{\et}{\mathrm{\acute{e}t}}
\DeclareMathOperator{\mult}{m}
\DeclareMathOperator{\loc}{ll}
\DeclareMathOperator{\len}{length}
\DeclareMathOperator{\BT}{BT}
\DeclareMathOperator{\DBT}{DBT}
\DeclareMathOperator{\Div}{div}

\newcommand{\field}[1]{\mathbf{#1}}

\newcommand{\Q}{\field{Q}}

\newcommand{\Zp}{\field{Z}_p}

\newcommand{\Z}{\field{Z}}
\newcommand{\F}{\field{F}}

\renewcommand{\AA}{\mathcal{A}}

\renewcommand{\P}{\field{P}}
\newcommand{\ord}{\mbox{ord}}
\newcommand{\EE} {\mathcal{E}}

\newcommand{\FF} {\mathcal{F}}
\newcommand{\DD} {\mathcal{D}}
\newcommand{\GG} {\mathcal{G}}
\newcommand{\ra}{\rightarrow}

\newcommand{\HH}{\mathcal{H}}

\newcommand{\mat}[4]{\left[\begin{array}{cc}#1 & #2 \\
                                         #3 & #4\end{array}\right]}

\newcommand{\inj}{\hookrightarrow}

\newcommand{\MM}{\mathcal{M}}

\newcommand{\beq}{\begin{displaymath}}
\newcommand{\eeq}{\end{displaymath}}
\newcommand{\beqn}{\begin{equation}}
\newcommand{\eeqn}{\end{equation}}

\newcommand{\D}{\mathbf{D}}

\theoremstyle{plain}
\newtheorem{thm}{Theorem}[section]
\newtheorem{prop}[thm]{Proposition}
\newtheorem{cor}[thm]{Corollary}

\theoremstyle{definition}
\newtheorem{defn}[thm]{Definition}

\newtheorem*{questions}{Questions}
\newtheorem*{problem}{Problem}

\theoremstyle{remark}
\newtheorem{rem}[thm]{Remark}

\makeatletter

\newcommand{\Rmnum}[1]{\expandafter\@slowromancap\romannumeral #1@}
\makeatother

\title{Random Dieudonn\'{e} modules, random $p$-divisible groups, and random curves over finite fields}
\author{Bryden Cais, Jordan S. Ellenberg, and David Zureick-Brown}

  \usepackage[pdftex]{epsfig}
  \usepackage[pdftex,colorlinks=true,%
              pdftitle={Random Dieudonn\'{e} modules, random $p$-divisible groups, and random curves over finite fields,}%
              pdfauthor={Bryden Cais, Jordan S. Ellenberg, and David Zureick-Brown}]{hyperref}
\usepackage[alphabetic,lite]{amsrefs} % for bibliography

\usepackage[alphabetic,lite]{amsrefs} % for bibliography

\begin{document}
\maketitle

\begin{abstract}
We describe a probability distribution on isomorphism classes of principally quasi-polarized $p$-divisible groups over a finite field $k$ of characteristic $p$ which can reasonably be thought of as ``uniform distribution," and we compute the distribution of various statistics ($p$-corank, $a$-number, etc.) of $p$-divisible groups drawn from this distribution.  It is then natural to ask to what extent the $p$-divisible groups attached to a randomly chosen hyperelliptic curve (resp. curve, resp. abelian variety) over $k$ are uniformly distributed in this sense.  This heuristic is analogous to conjectures of Cohen-Lenstra type for $\car k \neq p$, in which case  the random $p$-divisible group is defined by a random matrix recording the action of Frobenius.  Extensive numerical investigation reveals some cases of agreement with the heuristic and some interesting discrepancies.  For example, plane curves over $\F_3$ appear substantially less likely to be ordinary than hyperelliptic curves over $\F_3$.
\end{abstract}

\section{Introduction}

Let $q$ be a power of a prime $p$, and let $C/\F_q$ be a genus $g$ hyperelliptic curve with affine equation
\beq
y^2 = f(x)
\eeq
where $f$ is chosen {\em at random} from the set of monic squarefree polynomials of degree $2g+1$.  We can think of $\F_q(C)$ as a ``random quadratic extension of $\F_q (t)$," and ask about the probability distribution (if there is one) on arithmetic invariants of $C$.

For example: if $\ell$ is an odd prime not equal to $p$, one can ask
about the distribution of the $\ell$-primary part of the ideal class
group $\Cl(\F_q (C))$ -- or, equivalently, the group of $\F_q$-rational
points of the $\ell$-divisible group $J(C)[\ell^\infty]$.   The
distribution of $\Cl(\F_q (C))[\ell^\infty]$ is the subject of the {\em
  Cohen-Lenstra conjectures}~\cite{cohen1984heuristics},
which predict that the probability distribution on the isomorphism classes of $\Cl(\F_q (C))[\ell^\infty]$ approaches a limit, the so-called {\em Cohen-Lenstra distribution}, as $g \ra \infty$.  More precisely, Cohen and Lenstra proposed this conjecture for the class groups of quadratic number fields, but it was quickly understood (see e.g., \cite{friedman1989distribution}) that the underlying philosophy was just as valid for hyperelliptic function fields.

Much less effort has been devoted to the case where $\ell = p$,
perhaps because this question about function fields has no obvious
number field analogue.  Nonetheless, it is quite natural to ask
whether the $p$-adic invariants of random hyperelliptic curves over
$\F_q$ (or random curves, or random abelian varieties) obey statistical
regularities.  For example:  

\begin{quote}
What is the probability that a random hyperelliptic curve has ordinary Jacobian?
\end{quote}

More precisely:  let $P_o(q,d)$ be the proportion of the $q^d - q^{d-1}$ monic squarefree polynomials $f(x)$ over $\F_q$ of degree $d$ such that the curve $C_f$ with equation $y^2 = f(x)$ has ordinary Jacobian.  Then we ask:  does $\lim_{d \ra \infty} P_o(q,d)$ exist, and if so, what is its value?  Of course one can ask similar questions about other invariants of the $p$-divisible group of $\Jac(C_f)$, such as $a$-number, $p$-rank, Newton polygon, or final type.

One can interpret the Cohen-Lenstra conjecture as an assertion that
the $\ell$-divisible group of $\Jac(C_f)$ behaves like a ``random
principally polarized $\ell$-divisible group over $\F_q$."  (This
point of view begins with Friedman and
Washington~\cite{friedman1989distribution} and has subsequently been refined by
Achter~\cite{Achter:distributionClassGroups},
 Malle~\cite{Malle:distribution},
 and Garton~\cite{Garton:thesis}.)  A principally polarized $\ell$-divisible group of rank $2g$ over $\F_q$ is the same thing as an abelian group isomorphic to $(\Q_\ell / \Z_\ell)^{2g}$, equipped with a nondegenerate symplectic form $\omega$ and an automorphism $F$ satisfying $\omega(Fx,Fy) = q \omega(x,y)$.  In other words, $F$ is chosen from a certain coset of $\Sp_{2g}(\Z_\ell)$ in $\GSp_{2g}(\Z_\ell)$.  Choosing $F$ at random with respect to Haar measure then specifies a notion of ``random $\ell$-divisible group of rank $2g$ over $\F_q$," which allows us to compute notions like ``the probability that a principally polarized $\ell$-divisible group of rank $2g$ over $\F_q$ has no nontrivial $\F_q$-rational point."  Moreover, as $g$ goes to infinity, this probability approaches a limit; when $q$ is not congruent to $1$ mod $\ell$, the limit is
\beq
\prod_{i=1}^\infty (1-\ell^{-1})
\eeq
which is precisely the Cohen-Lenstra prediction for the probability that $\Jac(C_f)[\ell^\infty](\F_q)$ is trivial.

In the same way, one might ask whether the $p$-divisible group of
$\Jac(C_f)$ is in any sense a ``random principally quasi-polarized
$p$-divisible group over $k$."  The first task is to define this
notion.    A $p$-divisible
group over $\F_q$ is determined by its {\bf Dieudonn\'{e} module,} a
free $\Z_q:=W(\F_q)$-module with some extra ``semilinear algebra"
structures.  It turns out that there is a natural correspondence between principally quasi-polarized Dieudonn\'{e}
modules of rank $2g$ over $\Z_q$ a certain
double coset of $\Sp_{2g}(\Z_q)$ in the group of $\Z_p$-linear transformations of $\Z_q^{2g}$.  
\phantomsection
\label{TODO:28}
From this
description one obtains a probability measure on principally quasi-polarized
Dieudonn\'{e} modules, and thus on principally quasi-polarized $p$-divisible
groups over $k$.  In \S \ref{sec:randompdiv} we study the statistics of several natural invariants of random principally quasi-polarized $p$-divisible groups, and show that these approach limiting distributions as $g$ goes to infinity.  For example, we prove
\begin{prop}  The probability that a random principally quasi-polarized $p$-divisible group of rank $2g$ over $\F_q$ has $a$-number $r$ approaches
\beq
 q^{- {r +1 \choose 2}} \prod_{i=1}^\infty(1+q^{-i})^{-1} \prod_{i=1}^r (1-q^{-i})^{-1}.
\eeq
as $g \ra \infty$.  In particular, the probability that a random principally quasi-polarized $p$-divisible group of rank $2g$ over $\F_q$ is {\bf ordinary} approaches
\begin{equation}	
\prod_{i=1}^\infty(1+q^{-i})^{-1} = \prod_{j=1}^\infty (1-q^{1-2j}).
\label{eq:probord}
\end{equation}
as $g \ra \infty$.
\end{prop}

We call the infinite product in \eqref{eq:probord} the {\bf Malle-Garton constant}, since it is the same constant that occurs in the work of the two named authors on conjectures of Cohen-Lenstra type over fields containing a $p$th root of unity.

We also compute some statistics for the $p$-rank and the group of
$\F_q$-rational points of a random principally quasi-polarized $p$-divisible group $G$; for instance, we find that the probability that the $p$-rank of $G$ is $g-1$ (one smaller than maximum) is 
\beq
q^{-1} \prod_{i=1}^\infty(1+q^{-i})^{-1}
\eeq
and the probability that the group of $\F_q$-rational $p$-torsion
points has dimension $r$ (as an $\F_p$-vector space) is exactly the Cohen-Lenstra probability
\begin{equation}
q^{-r^2} \prod_{i=1}^r (1-q^{-i})^{-1} \prod_{j=r+1}^\infty (1-q^{-j}).
\label{eq:cohenlenstraintro}
\end{equation}

The notion of ``random $p$-divisible group" having been specified, it
remains to ask whether the $p$-divisible groups of random
hyperelliptic Jacobians act like random $p$-divisible groups.  In \S \ref{sec:experiments}, we gather some numerical evidence concerning
this question.  The results are in some sense affirmative, but display
several surprising (to us) features.  

For instance, the probability that a random hyperelliptic curve over $\F_3$ is ordinary does not appear to approach the value given in \eqref{eq:probord}.  Rather, it is apparently converging to $2/3$.  However, it does not seem that $\lim_{d \ra \infty} P_o(q,d) = 1-1/q$ in general.  For instance, $P_o(5,d)$ appears to be converging to $(1-1/5)(1-1/125) = 0.7936$, which is a truncation of the second infinite product in \eqref{eq:probord}.  For larger $q$, the difference between $P_o(q,d)$ and $\prod_{i=1}^\infty(1+q^{-i})^{-1}$ is too small to detect reliably from our data.  We have no principled basis to make a conjecture about the precise value $\lim_{d \ra \infty} P_o(q,d)$ but our data is certainly consistent with the hypothesis that the limit exists, and that
\begin{equation}
\frac{\log(\lim_{d \ra \infty} P_o(p,d) - \prod_{i=1}^\infty(1+p^{-i})^{-1})}{\log p}
\label{eq:approaching}
\end{equation}
goes to $-\infty$ as $p$ grows.  

One might speculate that the discrepancy between experiment and heuristic is a result of our restriction to hyperelliptic curves.  What if we consider random curves, or even random principally polarized abelian varieties, more generally?  The moduli spaces $\MM_g$ and $\AA_g$ are both of general type for large $g$, making it hopeless to sample curves or abelian varieties truly at random.  But one can at least study various rationally parametrized families.  We find that the proportion of ordinary {\em plane curves} over $\F_3$ is indistinguishable from $\prod_{i=1}^\infty(1+3^{-i})^{-1}$.  In other words, {\em plane curves over $\F_3$ are substantially less likely than hyperelliptic curves to be ordinary.}  We have no explanation for this phenomenon.  The proportion of ordinary plane curves over $\F_2$ does not seem to be $\prod_{i=1}^\infty(1+2^{-i})^{-1}$; the data is consistent, though, with the hypothesis that \eqref{eq:approaching} holds for plane curves as well as hyperelliptic curves.

In \S \ref{ss:geometry}, we discuss the geometry of various strata in the moduli space of hyperelliptic curves, and what relationship the statistical phenomena observed in \S \ref{sec:experiments} bear to the geometry of these spaces.

\subsection*{Acknowledgements}
We would like to thank Jeff Achter,  Derek Garton, Tim Holland, Bjorn
Poonen, and Rachel Pries for useful conversations.

\section{From $p$-divisible groups to Dieudonn\'{e} modules}
\label{s:background}

Let $k$ be a perfect field of characteristic $p$.
We briefly recall the classification of finite flat group schemes of $p$-power order and of $p$-divisible groups over $k$ afforded by (covariant) Dieudonn\'e theory.  Some general references on these topics
are \cite{Demazure:divisibleLecture}, \cite{Fontaine:divisible},
\cite[\Rmnum{2}--\Rmnum{3}]{Grothendieck:barsottiTate} and
\cite{Tate:pDivisible}. 
Readers already familiar with this story may skip immediately to the next section.

%While not logically necessary to the remainder of the paper, this section provides
%the geometric motivation and meaning for our subsequent semilinear algebra investigations.

Let $G$ be a finite commutative $k$-group scheme of $p$-power order.\footnote{The
{\em order} of $G$ is by definition the $\dim_k(A)$ where $G=\Spec(A)$.}
We denote by $G^{\vee}$ the Cartier dual of $G$, and note that one has a
canonical ``double duality" isomorphism $G^{\vee\vee}\simeq G$.
We write $F_G\colon G\rightarrow G^{(p)}$
and $V_G\colon G^{(p)}\rightarrow G$ for the relative Frobenius and Verscheibung morphisms,
respectively, and when $G$ is clear from context we will simply write 
$F^r$ and $V^r$ for the $r$-fold iterates of relative Frobenius and Verscheibung, defined in 
the obvious way.

Note that by definition, the composition of $F$ and $V$ in either order
is multiplication by $p$.
Since $k$ is perfect, there is a canonical decomposition of $G$
into its \'etale, multiplicative, and local-local subgroup schemes
\begin{equation}
		G= G^{\et}\times_k G^{\mult}\times_k G^{\loc},
	\label{ConEt}
\end{equation}
which is characterized by:
\begin{itemize}
	\item $G^{\et}$ is the maximal subgroup scheme of $G$ on which $F$ is
	an isomorphism.
	\item $G^{\mult}$ is the maximal subgroup scheme of $G$ on which $V$
	is an isomorphism.
	\item $G^{\loc}$ is the maximal subgroup scheme on which $F$ and $V$
	are both nilpotent (in the sense that $F^n=V^n=0$ for all $n$ sufficiently
	large).
\end{itemize}
For $\star\in \{\et,\mult,\loc\}$, 
the formation of $G^{\star}$ is functorial in $G$, and there are canonical 
identifications $(G^{\loc})^{\vee}\simeq (G^{\vee})^{\loc}$, 
$(G^{\mult})^{\vee}\simeq (G^{\vee})^{\et}$ and
$(G^{\et})^{\vee}\simeq (G^{\vee})^{\mult}$.

For a group $G$ of $p$-power order that is killed by $p$,
the nonnegative integers
\begin{equation*}
		f=f(G):=\log_p(\ord(G^{\et}))\qquad\text{and}\qquad
		a=a(G):=\dim_{k} \Hom_{\mathrm{gps}/k}(\alpha_p, G)
\end{equation*}
are called the {\bf $p$-rank} and {\bf $a$-number} of $G$, respectively,
where $\alpha_p=\Spec(k[X]/X^p)$ is the unique simple object in the category
of $p$-power order groups over $k$ which are killed by $p$
and are of local-local type.\footnote{In the definition of $a$-number, 
we view the $\Hom$ group as a left module over $\End(\alpha_p/k)=k$
in the obvious way.}
We say that $G$ is of {\bf $\alpha$-type}
if $G\simeq \alpha_p^m$ for some $m\ge 1$.

\begin{rem}\label{farems}
Let $G$ be a $k$-group of $p$-power order that is killed by $p$.
\begin{enumerate}
	\item  For {\em any} extension $k'/k$ contained in $\overline{k}$, one has 
	$a(G)=\dim_{k'}\Hom_{\mathrm{gps}/k'}(\alpha_p, G_{k'})$
	so the $a$-number is insensitive to algebraic extension of $k$.

	\item A finite $k$-group of $p$-power order that is killed by $p$
	 is of $\alpha$-type if and only if
	its relative Frobenius and Verscheibung morphisms are both zero.
	It follows easily from this that $G$ has a unique
	maximal subgroup of $\alpha$-type, which we denote by $G[F,V]$.
	We then have $G[F,V]\simeq \alpha_p^{a(G)}$, so 
	%(as one sees by choosing an $k$-basis of $\Hom_{\mathrm{gps}/k}(\alpha_p,G)$), the 
	$a(G)$ is the largest integer $m$ for which there exists a 
	closed immersion $\alpha_p^m\hookrightarrow G$ of $k$-group schemes.
	\label{alphatype}

	\item As the inclusion $G^{\et}(\overline{k})\hookrightarrow G(\overline{k})$ is an equality,
	the $p$-rank of $G$ is the nonnegative integer $f$ for which
	$|G(\overline{k})| = p^f$, and one will often see the $p$-rank of $G$
	defined this way.

	\item In the special case that $G\simeq G^{\vee}$, one has 
	$G^{\mult}\simeq (G^{\vee})^{\mult}\simeq (G^{\et})^{\vee}$, so also
	\begin{equation*}
		f(G)=\log_p(\ord(G^{\mult})) = 
		\dim_{\F_p}\Hom_{\mathrm{gps}/\overline{k}}(\mu_p, G_{\overline{k}}).
	\end{equation*}
	 The right side of the equation above is often used as a definition
	of $f(G)$ in the literature, since then the definitions for $p$-rank and $a$-number
	look more alike.
\end{enumerate}
\end{rem}

Attached to $G$ is its (covariant) {\bf Dieudonn\'e module}, $\D(G)$, which
is a finite length $W(k)$-module equipped with a $\sigma$-semilinear
additive map $F\colon D\rightarrow D$
and a $\sigma^{-1}$-semilinear additive map $V\colon D\rightarrow D$ which satisfy $FV=FV=p$.
We view $\D(G)$ as a module over the {\bf Dieudonn\'e ring}: this
is the (generally) noncommutative ring $A_k$ generated over $W(k)$ by two indeterminates 
$F$ and $V$ which satisfy the relations $F\lambda=\sigma(\lambda) F$, $V\lambda=\sigma^{-1}(\lambda)V$
and $FV=VF=p$ for all $\lambda\in W(k)$.

If $D$ is {\em any} (left) $A_k$-module of finite $W(k)$-length, we define the {\bf dual}
of $D$ to be the $A_k$-module
$D^{\vee} := \Hom_{W(k)}(D,W(k)[1/p]/W(k))$ with\footnote{
For any $\tau\in \Aut(k)$ and any $\tau$-semilinear additive map $\Psi:D\rightarrow D$,
we define the {\em dual} of $\Psi$ to be the $\tau^{-1}$-semilinear additive map
$\Psi^{\vee}:D^{\vee}\rightarrow D^{\vee}$ whose value on any linear functional
$L\in D^{\vee}$ is the linear functional $\Psi^{\vee}(L):=\tau^{-1} \circ L\circ\Psi$.
}
$F_{D^{\vee}}:=V^{\vee}_D$ and $V_{D^{\vee}}=F^{\vee}_D$.
One checks that the expected double duality isomorphism $D^{\vee\vee}\simeq D$ (as left $A_k$-modules) holds.
The main theorem of classical Dieudonn\'e theory is:

\begin{thm}\label{DieudonneMain}
	The functor $G\rightsquigarrow \D(G)$ from the category of commutative finite $k$-group schemes
	of $p$-power order to the category of left $A_k$-modules of finite $W(k)$-length is an 	
	exact equivalence
	of abelian categories.  Moreover:
	\begin{enumerate}
		\item The order of $G$ is $p^v$ where $v=\len_{W(k)}\D(G)$.\label{ordercalc}
		
		\item There is a natural isomorphism of left $A_k$-modules $D(G^{\vee})\simeq D(G)^{\vee}$.

		\item The canonical decomposition $G= G^{\mult}\times G^{\et}\times G^{\loc}$
		corresponds to the canonical decomposition of $D:=\D(G)$
		\begin{equation*}
				D = D^{\et}\oplus D^{\mult}\oplus D^{\loc}
		\end{equation*}
		where $D^{\star}:=\D(G^{\star})$ for $\star\in \{\et,\mult,\loc\}$.  These $A_k$-submodules
		are characterized by: 
		\begin{itemize}
			\item $D^{\et}$ is the maximal submodule of $D$ on which $F$ is bijective.
			%$($and necessarily $V$ is nilpotent$)$.
			\item $D^{\mult}$ is the maximal submodule of $D$ on which $V$ is bijective.
			%$($and necessarily $F$ is nilpotent$)$.
			\item $D^{\loc}$ is the maximal submodule of $D$ on which both $F$ and $V$
			are nilpotent.
		\end{itemize}
		
		\item $\D(\alpha_p)=k$ with $F=V=0$.  In particular, if $G$ is killed by $p$
		then $\D(G[F,V])$ is identified with the intersection 
		$\ker F\cap \ker V\subseteq \D(G)$, so $a(G)=\dim_k(\ker F\cap \ker V)$.\label{alphacalc}
	\end{enumerate}
\end{thm}

By definition, a {\bf $p$-divisible $(=$ Barsotti-Tate$)$ group over $k$ of height $h$}
is an inductive system $G:=\{(G_{v},i_v)\}_{v\ge 0}$ of finite $k$-group
schemes, with $G_v$ of order $p^{hv}$, such that 
\begin{equation}
	\xymatrix{
		0\ar[r] & {G_v} \ar[r]^-{i_v} & {G_{v+1}} \ar[r]^-{p^v} & {G_{v+1}} \ar[r] & 0
		}\label{pDivDef}
\end{equation}
is an exact sequence of $k$-group schemes for all $v\ge 0$.  The two prototypical examples
are $\Q_p/\Z_p:=(\Z/p^v\Z,i_v)$ and $\mu_{p^{\infty}}:=(\mu_{p^v},i_v)$,
with $i_v$ the obvious closed immersions in each case.  We will be primarily interested
in the $p$-divisible groups associated to abelian varieties: If $A/k$
is an abelian variety, then $A[p^{\infty}]:=(A[p^v],i_v)$ is naturally a $p$-divisible
group, with $i_v$ the canonical closed immersion.
We recall that the {\bf dual} of $G$ is by definition
the inductive system $G^{\vee}:=(G_v^{\vee},j_v^{\vee})$
with $j_v\colon G_{v+1}\rightarrow G_v$ the unique map\footnote{This map exists
as the exactness of (\ref{pDivDef}) forces the multiplication by $p$ map
on $G_{v+1}$ to factor through $i_v$.} satisfying $i_v\circ j_v = p$.
The double-duality isomorphisms at finite level compile to give a canonical
isomorphism of $p$-divisible groups $G^{\vee\vee}\simeq G$ over $k$.
%For example, $\mu_{p^{\infty}}$ and $\Q_p/\Z_p$ are dual.
Similarly, the decomposition (\ref{ConEt}) induces a corresponding splitting
$G = G^{\et}\times_k G^{\mult} \times_k G^{\loc}$
with each $G^{\star}$ the inductive system of the $G_v^{\star}$,
and we say that $G$ is \'etale if $G=G^{\et}$ and so forth.

Furthermore, there are natural identifications
$(G^{\loc})^{\vee} = (G^{\vee})^{\loc}$,
$(G^{\mult})^{\vee} \simeq (G^{\vee})^{\et}$ and
$(G^{\et})^{\vee} \simeq (G^{\vee})^{\mult}$
induced by the ones at finite level.

A {\bf principal quasi-polarization} of a $p$-divisible group $G$ is an isomorphism 
$
\lambda\colon G\xrightarrow{\simeq}  G^{\vee}
$
with the property that the composition $G^{\vee\vee}\rightarrow G^{\vee}$
of the canonical double duality map $G^{\vee\vee}\simeq G$ with $\lambda$

coincides with $-\lambda^{\vee}$.  We say that $G$ is {\bf principally quasi-polarized},
or {\bf pqp} for short, if it is endowed with a principal quasi-polarization.
If $G=A[p^{\infty}]$ for an abelian variety $A$, then any polarization of degree prime to $p$
on $A$ induces a principal quasi-polarization on $G$.

The {\bf Dieudonn\'e module} of a $p$-divisible group $G:=\{(G_v,i_v)\}$ is
by definition the $A_k$-module $\D(G):=\varprojlim_{j_v} \D(G_v)$.

It follows easily from Theorem \ref{DieudonneMain} (\ref{ordercalc}) and definitions that 
$\D(G)$ is free of rank $h$ as a $W(k)$-module.
If $D$ is {\em any} (left) $A_k$-module that is $W(k)$-finite and free, we define the {\bf dual}
of $D$ to be the $A_k$-module
$D^{\vee}:=\Hom_{W(k)}(D,W(k))$ with $F_{D^{\vee}}:=V_D^{\vee}$
and $V_{D^{\vee}}:=F_D^{\vee}$, and we define the {\bf $p$-rank and $a$-number} of $G$
to be the corresponding $p$-rank and $a$-number of $G[p]:=G_1$.
From Theorem \ref{DieudonneMain} we obtain:

\begin{thm}\label{pdivequiv}
	The functor $G\rightsquigarrow \D(G)$ from the category of $p$-divisible groups
	over $k$ to the category of left $A_k$-modules which are finite and free over $W(k)$
	is an equivalence of categories.
	If $G=\{(G_v,i_v)\}$ is a $p$-divisible group of height $h$ with Dieudonn\'e module
	$D=\D(G)$ then:
	\begin{enumerate}
		\item $D$ is a free $W(k)$-module of rank $h$, and $F$, $V$
		uniquely determine each other.
				
		\item There is a natural isomorphism of left $A_k$-modules 
		$\D(G^{\vee})\simeq D^{\vee}$.

	%	\item $G$ is \'etale $($respectively connected, resp. co-\'etale $[$$=$multiplicative$]$, 
	%	resp. co-connected$)$
	%	if and only if $F$ is bijective $($resp. $F$ is topologically nilpotent, resp. $V$ is bijective, resp. 
	%	$V$ is topologically nilpotent$)$.  In particular: 
		
		\item The canonical decomposition $G= G^{\mult}\times G^{\et}\times G^{\loc}$
		corresponds to the canonical decomposition of $D$
		\begin{equation*}
				D = D^{\et}\oplus D^{\mult}\oplus D^{\loc}
		\end{equation*}
		where $D^{\star}:=\D(G^{\star})$ for $\star\in \{\et,\mult,\loc\}$. These $A_k$-submodules
		are characterized by
		\begin{itemize}
			\item $D^{\et}$ is the maximal submodule of $D$ on which $F$ is bijective.
			%$($and necessarily $V$ is topologically nilpotent$)$.
			\item $D^{\mult}$ is the maximal submodule of $D$ on which $V$ is bijective.
			%$($and necessarily $F$ is topologically nilpotent$)$.
			\item $D^{\loc}$ is the maximal submodule of $D$ on which $F$ and $V$
			are topologically nilpotent.
		\end{itemize}
		
		\item $G$ is pqp if and only if
		there exists a symplectic $(=$ perfect, bilinear, alternating$)$ form
		\begin{equation*}
			\xymatrix{
			{\psi_G \colon D\times D} \ar[r] & {W(k)}
			}\quad\text{satisfying}\quad
			\psi_G(Fx,y) = \sigma(\psi_G(x,Vy)).\label{symplectic}
		\end{equation*}
	
		\item Set $\overline{D}=D/pD=\D(G[p])$ and let $\overline{F}:=F\bmod p$ and $\overline{V}:=V\bmod p$.
		Then the $p$-rank of $G$ is the ``infinity rank" of $\overline{F}$, i.e.
		the $k$-dimension of the maximal subspace of $\overline{D}$ on which $\overline{F}$
		is bijective.  The $a$-number of $G$ is the $k$-dimension of
		$\ker \overline{F}\cap \ker \overline{V}$.
	\end{enumerate}
\end{thm}

In view of Theorems \ref{DieudonneMain} and \ref{pdivequiv}, we will often abuse terminology
and say that a Dieudonn\'e module has some property if the corresponding
$p$-divisible group or finite group scheme has this property, and vice-versa.

We will be interested in the category $\BT_1$ of commutative group schemes $G$ over $k$
of $p$-power order that are killed by $p$ which arise as the $p$-torsion in a Barsotti-Tate
group over $k$.  This is equivalent to the condition that the sequence
\begin{equation} 
	\xymatrix{
		G \ar[r]^-{F_G} & {G^{(p)}} \ar[r]^-{V_G} & {G}
	}\label{BTcondition}
\end{equation}
is exact.  An object of $\BT_1$ is called a {\bf truncated Barsotti-Tate group
of level $1$}, or a $\BT_1$ for short.  We will write $\DBT_1$ for the
category of finite $\F_q$-vector spaces equipped with semilinear additive maps
$F$ and $V$ satisfying $FV=VF=0$ as well as $\im(F)=\ker(V)$ and $\ker(F)=\im(V)$.
We note that the Dieudonn\'e module functor restricts to an equivalence of
categories $\BT_1\rightarrow \DBT_1$, and that, for example,
$\mu_p$ and $\Z/p\Z$ are $\BT_1$'s whereas $\alpha_p$ is not.
A {\bf principal quasi-polarization of a $\BT_1$}
is a homomorphism $\lambda\colon G\rightarrow G^{\vee}$ with the property that the induced
bilinear form $\psi\colon\D(G)\times \D(G)\rightarrow k$ on the Dieudonn\'e module of $G$ is symplectic.
A {\bf pqp $\DBT_1$} is simply the Dieudonn\'e module of a principally quasi-polarized $\BT_1$.

\begin{rem}\label{pqprem}
The condition that $\psi$ is symplectic implies that $\lambda$
is an anti-selfdual isomorphism, {\em i.e.} that $\lambda^{\vee}$
coincides with $-\lambda$ via the double duality identification $G^{\vee\vee}\simeq G$.
If $\mathrm{char}(k)\neq 2$, then these two conditions are in fact equivalent.
In characteristic 2, however, there exist anti-selfdual isomorphisms
$\lambda\colon G\rightarrow G^{\vee}$ which {\em do not} induce a symplectic form
on the Dieudonn\'e module, basically because there are (skew) symmetric forms
in characteristic 2 which are not alternating.  In general, if $\widetilde{G}$
is a $p$-divisible group over $k$ prolonging $G$, then any principal
quasi-polarization on $\widetilde{G}$ induces a principal quasi-polarization on $G$.
Our definition of a principal quasi-polarization on a $\BT_1$ is identical to that found in
\cite[\S2.6]{moonen:groupSchemes} and \cite[\S9.2]{oort:stratification}.

\end{rem}

\section{Statistics of random Dieudonn\'{e} modules}

\label{sec:randompdiv}

In this section we define a probability distribution on the isomorphism classes of pqp 
Dieudonn\'e modules $(D,F,V,\omega)$ of rank $2g$ over $\Z_q$.
To this end, we fix $D=\Z_q^{2g}$, with standard basis 
$\{e_1,\ldots,e_{2g}\}$ and corresponding standard symplectic form
$\omega$, and we will think of $F$ and $V$ as the entities to be chosen randomly. We will say that a $\sigma$-semilinear endomorphism
$F\colon D\rightarrow D$ is {\bf $p$-autodual} if there exists 
a $\sigma^{-1}$-semilinear endomorphism $V\colon D\rightarrow D$
such that the quadruple $(D,F,V,\omega)$ is a pqp Dieudonn\'e module,
or equivalently if $FV=VF=p$ and $\omega(Fx,y)=\sigma(\omega(x,Vy))$ for every $x,y \in D$.
Observe that $V$ is uniquely determined by $F$, if it exists.

For simplicity of notation we say ``$\sigma$-endomorphism" (resp. ``$\sigma$-automorphism") to mean ``$\sigma$-semilinear endormorphism" (resp. ``$\sigma$-semilinear automorphism").
Suppose given a $p$-autodual endomorphism $F$ of $D$.
We note first of all that $FD \subset D$ is a subgroup of index $q^g$
containing $pD$; in particular, the quotient $FD/pD$ is a subgroup of $D/pD$  isomorphic to $(\F_q)^g$.  For any $x,y \in D$, we have
\beq
\omega(Fx,Fy) = \sigma(\omega(x,VFy)) = p\sigma(\omega(x,y))
\eeq
so that $FD/pD$ is an isotropic (whence maximal isotropic) subspace of the symplectic $\F_q$-vector space $D/pD$.  We denote $D/pD$ by $W$.

Let $F_0$ be a $p$-autodual endomorphism of $D$.  We define $\FF(D)$ to be the double coset $\Sp(D) F_0 \Sp(D)$, where $\Sp(D)$ is the group of ($\Z_q$-linear) symplectic similitudes.   Then $\FF(D)$ is endowed with a probabilty measure by pushforward from Haar measure on the group $\Sp(D) \times \Sp(D)$.  Note that every element $F$ of $\FF(D)$ is in fact a $p$-autodual $\sigma$-endomorphism of $D$.  We now show that {\em all} $p$-autodual $\sigma$-endomorphisms arise in this way, which means that $\FF(D)$ is independent of our original choice of $F_0$.

\begin{prop} $\FF(D)$ is the set of $p$-autodual $\sigma$-endomorphisms of $D$.
\end{prop}

\begin{proof}
Let $F$ be a $p$-autodual $\sigma$-endomorphism of $D$.
Then $FD/pD$ is a maximal isotropic subspace of $W$.  By the
symplectic version of Witt's
Theorem ~\cite{Artin:geometricAlgebra}*{Theorem 3.9}, $\Sp(W)$ acts
transitively on the maximal isotropic subspaces and since $\Sp(D)$
surjects onto $\Sp(W)$, we can choose $g \in \Sp(D)$ such that $g F_0
D / pD = F D / pD$.  Since $FD$ and $g F_0 D$ both contain $pD$, these
two subgroups of $D$ are actually equal.  Now $F$ induces an isomorphism from $D$ to $FD$, so its inverse $F^{-1}$ can be thought of as an isomorphism from $FD$ to $D$.  Thus the composition $F^{-1} g F_0$ is a $\Z_q$-linear automorphism of $D$ which preserves $\omega$; i.e., it is an element $g'$ of $\Sp(D)$.  This shows that $F$ lies in $\Sp(D) F_0 \Sp(D)$, as claimed.
\end{proof}

By a {\bf random pqp Dieudonn\'{e} module $(D,F,V,\omega)$ of
  dimension $2g$} we mean one in which $F$ is chosen randomly from
$\FF(D)$ with respect to the above probability measure, and $V$ is determined from $F$.  When there is no danger of confusion we denote $(D,F,V,\omega)$ simply by $D$.

  If $X$ is a statistic of Dieudonn\'{e} modules (such as $a$-number, or dimension of local-local part) we denote by $\EE_g(X)$ the expected value of $X(D)$ where $D$ is a random pqp Dieudonn\'{e} module of rank $2g$.  The statistics of interest to us are those where $\EE_g(X)$ approaches a limit as $g \ra \infty$; in this case we denote the limit by $\EE(X)$ and refer to it as the expected value of $X$ for a random pqp Dieudonn\'{e} module, the ``large $g$ limit" being understood.  If $P$ is a true-or-false assertion about Dieudonn\'{e} modules, the ``probability that a random Dieudonn\'{e} module satisfies $P$" is understood to mean the expected value of the Bernoulli variable which is $1$ when $P$ holds and $0$ otherwise.

We define $\overline{\FF}(D)$ to be the reduction of $\FF(D)$ modulo $p$, with
its inherited probability distribution.  In other words, if
$\overline{F}_0$ is the reduction of $F_0$ to a $\sigma$-endomorphism of
$W$, then $\overline{\FF}(D)$ is $\Sp(W) \overline{F}_0 \Sp(W)$, with the
probability distribution obtained by pushforward from the Haar measure
(i.e., the counting measure) on $\Sp(W) \times \Sp(W)$.  If $F$ is an
element of $\overline{\FF}(D)$, then $F$ is a $\sigma$-endomorphism of $W$
whose kernel is a maximal isotropic subspace of $W$.  The reduction
$\overline{\omega}$ of $\omega$ provides an isomorphism $\lambda:W\rightarrow W^{\vee}$ between $W$ and its
$\F_q$-linear dual, and we set $V:=\lambda^{-1}\circ F^{\vee}\circ\lambda$.
Note $V$ is then a $\sigma^{-1}$-endomorphism of $W$ satisfying $VF=FV = 0$, 
and that in fact $(W,F,V,\overline{\omega})$ is a pqp $\DBT_1$.

By a  {\bf random pqp $\DBT_1$} we mean a $\DBT_1$ $(W,F,V)$, endowed with the sympletic form $\overline{\omega}$, obtained as above by choosing a random element of $\overline{\FF}(D)$.  Statistics of random pqp $\DBT_1$ are again understood to be computed in the large $g$ limit.

In this paper, we will restrict our attention almost entirely to
invariants of $p$-divisible groups which depend only on the associated
$\DBT_1$, such as $a$-number and $p$-corank.  It should certainly be
possible to compute the statistics of more refined invariants
(e.g., Newton polygon) but with the aim of avoiding ungrounded
speculation in the context of abelian varieties, we have mostly restricted ourselves to invariants for
which we have collected substantial experimental data on Jacobians on curves.

The kernel and image of $F$ are maximal isotropic subspaces of $W$ of
dimension $g$; we denote by $F'$ the $\sigma$-isomorphism from $W/\ker
F$ to $FW$ induced by $F$.  The following proposition provides a
useful description of a random  pqp $\DBT_1$ in terms of $\ker F, \im
F$, and $F'$.

\begin{prop}  Let $(W,F,V,\overline{\omega})$ be a random pqp $\DBT_1$.  Then $(\ker F,\im F)$ is uniformly distributed on the set of pairs of maximal isotropic subspaces of $W$, and $F'$ is uniformly distributed among $\sigma$-isomorphisms from $(W / W_1)$ to $W_2$.

\label{pr:uniformf}
\end{prop}

\begin{proof}  
The action of $\Sp(W) \times \Sp(W)$ is transitive on pairs of maximal isotropic subspaces, and the probability distribution on $\overline{\FF}(D)$ is invariant under this action; this gives the first assertion.  Now suppose that we condition on $\ker F = W_1$ and $\im F = W_2$; let $\overline{\FF}(D,W_1,W_2)$ be the subset of $\overline{\FF}(D)$ satisfying this condition.  Then $\overline{\FF}(D,W_1,W_2)$ is still invariant under left multiplication by the subgroup of $\Sp(W)$ preserving $W_2$; this subgroup is in fact isomorphic to $\GL(W_2)$ and permutes the choices of $F'$ transitively.  This yields the second assertion.

\end{proof}

The definitions given here may seem somewhat unsatisfactory;  our ``random $\DBT_1$" is in some sense more like ``a random $\DBT_1$ with a choice of $\Z_q$-basis."  We show below that our definition conforms with a more intrinsic definition of random $\DBT_1$. The groupoid formalism used here will not return again until the proof of Proposition~\ref{P:connComp}.

\begin{defn}  Let $G$ be a finite groupoid; that is, $G$ is a groupoid with finitely many isomorphism classes of objects and finite Hom sets.  The {\em uniform distribution} on $G$ is the unique distribution on isomorphism classes of objects whose mass on an isomorphism class $c$ is inversely proportional to the number of automorphisms of an object in $c$.  We say a finite set of objects $S$ in $G$ is {\em uniformly distributed} in $G$ if the probability that a random element of $S$ lies in an isomorphism class $c$ is given by uniform measure.
\label{de:uniformgroupoid}
\end{defn}

The desirability of counting objects with weights inversely proportional to the size of their automorphism group has been known at least since Siegel's mass formula; as regards general groupoids we learned the formalism from Baez and Dolan.

It is clear that an equivalence between groupoids $G_1$ and $G_2$ carries uniform measure on $G_1$ to uniform measure on $G_2$.  (This is just the groupoid version of the fact that a bijection of sets transports counting measure from one to the other.)  Similarly, if $S$ is a finite set with a {\em free} action of a group $\Gamma$, the pushforward from $S$ to $S/\Gamma$ of uniform measure on $S$ is uniform measure on $S/\Gamma$.  The following easy proposition records the fact that the same is true in the groupoid setting.

\begin{prop}  Let $S$ be a finite set, let $\Gamma$ be a finite group acting on $S$, and let $S/\Gamma$ be the groupoid whose objects are $S$ and whose morphisms from $s$ to $s'$ are group elements $\gamma$ in $\Gamma$ such that $\gamma \cdot s = s'$.  Then the objects of $S$ are uniformly distributed in $S/\Gamma$.
\label{p:groupoidquotient}
\end{prop}

\begin{proof}  The probability that a random $s$ in $S$ lies in the $\Gamma$-orbit $\Gamma s_0$ of a fixed $s_0 \in S$ is precisely
\beq
|\Gamma s_0| / |S| = 1 / \Aut_{S/\Gamma}(s_0).
\eeq
\end{proof}

We now explain how this formalism applies in the present context.  Let
$(W,\overline{\omega})$ be a $\F_q$-vector space of dimension $2g$ endowed with a nondegenerate symplectic form.  Let $S$ be the set of $\sigma$-endomorphisms $F\colon W \ra W$ whose kernel is a maximal isotropic subspace of $W$.  Note that $S$ is in bijection with the set of triples $(W_1, W_2, F')$ described in Proposition~\ref{pr:uniformf}; in particular, a random $\DBT_1$ of rank $2g$ is the same thing as a random element of $S$ in uniform distribution.

Now the group $\Gamma = \Sp_{2g}(\F_q)$ acts on $S$ by changes of basis preserving the symplectic form.  And the groupoid $S/\Gamma$ is equivalent to the category of (Dieudonn\'e modules of)
principally quasi-polarized $\BT_1$'s.
Thus, a random pqp $\DBT_1$, in our sense, is a random principally polarized $\DBT_1$ in the sense of Definition~\ref{de:uniformgroupoid}.

The above discussion is rather formal, but we will see that the groupoid viewpoint is quite convenient in the proof of Proposition~\ref{P:connComp}.

\subsection*{The a-number of a random pqp $\DBT_1$}

Let $(D,F,V)$ be a $\DBT_1$ over $\F_q$.
By Theorem \ref{DieudonneMain} (\ref{alphacalc}), the $a$-number
of $D$ is the $k$-dimension of the intersection $\ker F\cap \ker V\subseteq D$;
by definition of the category $\DBT_1$ we have $\ker V=\im F$,
so also $a(D)=\dim_k(\ker F\cap \im F)$.

\begin{prop}  
\label{P:aNumber}
The probability that $a(D) = r$ is
\beq
 q^{- {r +1 \choose 2}} \prod_{i=1}^\infty(1+q^{-i})^{-1} \prod_{i=1}^r (1-q^{-i})^{-1}.
\eeq
\label{pr:anumber}
\end{prop}

\begin{proof} The $a$-number does not depend on $F'$, so we are computing the probability that two random maximal isotropic subspaces of a large symplectic space over $\F_q$ intersect in a subspace of dimension $r$.  

Let $W$ be a $2g$-dimensional symplectic space.  The number of maximal isotropic subspaces in $W$ is
\begin{equation}
\frac{|\Sp_{2g}(\F_q)|}{q^{(1/2)g(g+1)} |\GL_g(\F_q)|} =q^{1/2(g^2+g)} \frac{|\Sp_{2g}(\F_p)|}{q^{2g^2 + g}} \frac {q^{g^2}}{|\GL_g(\F_q)|}
\label{eq:maxisocount}
\end{equation}

By the symplectic version of Witt's Theorem, the symplectic group $\Sp(W)$ acts transitively on the pairs of maximal isotropic subspaces with $r$-dimensional intersection; so to count the number of such pairs, it suffices to compute the size of the stabilizer of such a pair in $\Sp(W)$.  Suppose for instance that the pair is given by
\beq
V_1 = \langle e_1, \ldots, e_g \rangle, V_2 = \langle e_1, \ldots, e_r, e_{g+r+1}, \ldots, e_{2g} \rangle.
\eeq

Then the stabilizer of the pair $(V_1,V_2)$ is the group of matrices of the form
\beq
\mat{A}{B}{0}{(A^T)^{-1}}
\eeq
where $A$ lies in the parabolic subgroup preserving $\langle{e_1,
  \ldots, e_r \rangle}$ and $B$ is symmetric, having zero $(i,j)$
entry when $i > r$ and $j > g+r$.  

The order of this group is
\beq
q^{g^2 + {r +1 \choose 2}} \frac{|\GL_r(\F_q)|}{q^{r^2}} \frac{|\GL_{g-r}(\F_q)|}{q^{(g-r)^2}}
\eeq
so the number of pairs of maximal isotropics with $r$-dimensional intersection is
\begin{equation}
q^{g^2 + g - {r+1 \choose 2}} \frac{q^{r^2}}{|\GL_r(\F_q)|} \frac{q^{(g-r)^2}}{|\GL_{g-r}(\F_q)|} \frac{|\Sp_{2g}(\F_q)|}{q^{2g^2 + g}}
\label{eq:intrcount}.
\end{equation}
Dividing \eqref{eq:intrcount} by the square of \eqref{eq:maxisocount} yields
\beq
 q^{- {r +1 \choose 2}} \frac{q^{r^2}}{|\GL_r(\F_p)|} \frac{q^{(g-r)^2}}{|\GL_{g-r}(\F_q)|} \left( \frac{|\GL_g(\F_p)|}{q^{g^2}} \right)^2
 \frac{q^{2g^2+g}}{|\Sp_{2g}(\F_q)|}
\eeq
which, as $g$ goes to infinity with $r$ fixed, approaches the desired quantity.
\end{proof}

\begin{rem}
\label{r:postANumber}

  \begin{itemize}
  \item []
  \item[(1)] We note that this prediction is in keeping with the fact that
    the locus of abelian varieties with $a$-number at least $r$ in
    $\AA_g / \F_q$ has codimension $r+1 \choose 2$.

  \item[(2)] The work of Poonen and Rains \cite{PoonenR:maximalIsotropics}
    posits that the mod $p$ Selmer group of a random quadratic twist
    of a fixed elliptic curve should be distributed like the
    intersection of two random maximal isotropics in an {\em
      orthogonal} vector space.  They show that the mod $p$ Selmer
    group actually {\em does} arise as the intersection of two maximal
    isotropics -- the question, then, is whether these isotropics are
    in fact ``uniformly distributed" in an appropriate sense.  Our
    situation is similar; the $a$-number of a pqp $p$-divisible group
    is indeed the dimension of the intersection of the two maximal
    isotropics $FW$ and $VW$ in the symplectic vector space $W$, and
    one is asking whether these maximal isotropics are distributed
    uniformly when $W$ arises from an abelian variety.

  \item[(3)] The conjectured distribution of the $a$-number is the same as
    the distribution on the dimension of the fixed space of a random
    large symplectic matrix over $\F_q$, which was computed in an
    unpublished work by Rudvalis and Shinoda
    ~\cite{RudvalisS:enumeration} (see
    \cite{Fulman:conjugacySymplecticOrthogonal} for a review of their
    results and an alternative proof).  This distribution also appears
    in the conjectures of Malle~\cite{Malle:distribution} and
    Garton~\cite{Garton:thesis} as the conjectured distribution of
    $p$-ranks of ideal class groups of number fields containing $p$th
    roots of unity.  In the class group context, the relationship with
    the fixed space of a random symplectic matrix is motivated by the
    analogy between number fields and function fields, where the
    symplectic matrix describes the action of $\Frob_{\ell}$ on $p$-adic cohomology.

  \end{itemize}

\end{rem}

The $a$-number of $D$ is $0$ if and only if $D$ is ordinary.  We
thus have the following corollary.

\begin{cor}
The probability that a random pqp Dieudonn\'e module is ordinary is
\beq
\prod_{i=1}^\infty(1+q^{-i})^{-1}.
\eeq
\end{cor}

In problems of Cohen-Lenstra type, it is often the case that moments of variables have nicer formulae than probability distributions do.  The present situation is no exception.

\begin{prop}  Let $X_m(D)$ be the number of closed immersions of group
  schemes over $\F_q$ from $\alpha_p^m$ to the 
$p$-torsion in the $p$-divisible group associated to $D$.  Then $\EE X_m(D) = q^{m \choose 2}$.
\end{prop}

\begin{proof}
In the language of the present paper we claim that 
\beq
X_m(D) = (q^{a(D)} - 1) (q^{a(D)} - q) \ldots (q^{a(D)} - q^{m-1}).
\eeq
Indeed, if $G$ is the $p$-divisible group attached to $D$
then any closed immersion $\alpha_p^m\hookrightarrow G[p]$ of group schemes necessarily
factors through the maximal $\alpha$-type subgroup scheme $G[F,V]  \simeq \alpha_p^{a(G)}$ of $G[p]$
(see Remark \ref{farems} (\ref{alphatype})).  In particular, such closed immersions
are in bijection with closed immersions $\alpha_p^m\hookrightarrow \alpha_p^{a(G)}$,
which are in bijection with injections $\F_q^m\hookrightarrow \F_q^{a(G)}$
via the exact functor $\D(\cdot)$, thanks to Theorem \ref{DieudonneMain} (\ref{alphacalc}).   

Because the distribution of $X$ agrees with the distribution on the
dimension of the fixed space of a random matrix $g$ in $\Sp(W)$ (see
Remark \ref{r:postANumber} (3)), it
suffices to show that the number of injections from an $m$-dimensional
vector space into the fixed space of $g$ has the desired expected
value.  
By Burnside's Lemma, this is the same as the number of orbits of
$\Sp(W)$ acting on the set of injections $i\colon  \F_q^m \inj W$.  By
the symplectic version of Witt's Theorem,
%~\cite{Artin:geometricAlgebra}*{Theorem 3.9},
 two such injections $i_1,i_2$ are in the same orbit if the symplectic forms $i_1^*\langle \rangle$ and $i_2^*\langle \rangle$ agree; so the number of orbits is just the number of alternating bilinear forms on an $m$-dimensional vector space, which is  $q^{m \choose 2}$ as claimed.
\end{proof}
 
\subsection*{The $p$-corank of a random pqp $\DBT_1$}
\label{ss:pCorank}

Suppose that $X(D)$ is a statistic which is invariant under symplectic
change of basis, i.e., under conjugation of $F$ by $\Sp(W)$.  As above,
by Witt's Theorem
all pairs of maximal isotropic subspaces with intersection dimension
$r$ are in the same orbit of the symplectic group.  Thus, to compute the
expected value of $X(D)$ conditional on $a(D) = r$, it suffices to
compute the expected value of $X(D)$ for a fixed choice of $W_1$ and
$W_2$, and $F'$ chosen uniformly from the $\sigma$-isomorphisms
from $W / W_1$ to $W_2$.  The composition
\beq
\xymatrix@C=30pt{
	{\phi \colon W_2} \ar[r] & {W} \ar@{->>}[r] & {W/W_1} \ar[r]^-{F'} & {W_2}
}
\eeq
is then a $\sigma$-endomorphism of $W_2$ of rank $g-r$, and in 
fact is chosen uniformly from the set of such $\sigma$-endomorphisms.

When $W_2 = \im F$, the $\sigma$-endomorphism $\phi$ is just the map $FW \ra
F W$ 
induced by $F$.  In particular, the $p$-corank of the $p$-divisible group attached to $D$ is precisely the corank of $\phi^\infty$.

\begin{prop}
\label{p:corank}
Let $0 \leq r \leq s$ be integers. Then the probability that the a-number of $D$ is $r$ and the $p$-corank of $D$ is $s$ is
\begin{equation}
q^{-{r+1 \choose 2} + r - s} \prod_{i=1}^\infty (1+q^{-i})^{-1} \prod_{i=r}^{s-1} (1-q^{-i}) \prod_{i=1}^{s-r} (1-q^{-i})^{-1}.
\label{eq:corank}
\end{equation}
\label{pr:corank}
\end{prop}

\begin{proof}
We first show that the probability that a random $\sigma$-endomorphism
of a $g$-dimensional vector space $V$ has rank $g-r$ approaches
\begin{equation}
\label{eq:corankr}
q^{-r^2} \prod_{i=r+1}^\infty (1-q^{-i}) \prod_{i=1}^r (1-q^{-i})^{-1}
\end{equation}
as $g$ goes to infinity.
Indeed, the map $\phi \mapsto
(V \ra \im \phi, \im \phi)$ defines a bijection between $\sigma$-endomorphisms $\phi
\colon V \ra V$ of rank $g-r$ and pairs $(\psi,W)$ where $W \in
\Gr_{g,g-r}(\F_q)$ is a subspace of $V$ of dimension $g-r$ and
$\psi\colon V \to W$ is a $\sigma$-semilinear surjection. For a given such $W$, there are 
\begin{equation}
\label{eq:semi}
\prod_{i = 0}^{g-r-1}(q^g-q^i)
\end{equation}
such $\sigma$-semilinear surjections (indeed, the $\sigma$-semilinear surjective maps
$V\rightarrow W$ correspond to surjective {\em linear} mappings $\sigma^*(V)\rightarrow W$,
and $\sigma^*(V)$ is an $\F_q$-vector space of dimension $g$).
Considering the stabilizer of the natural action of $\GL_g$ on
$\Gr_{g,g-r}$ shows  that 
\begin{equation}
\label{eq:grass}
|\Gr_{g,g-r}(\F_q)| = 
\frac
{|\GL_g(\F_q)|}
{|\GL_r(\F_q)| |\GL_{g-r}(\F_q)| |\M_{r,g-r}(\F_q)|}.
\end{equation}
Dividing the product of (\ref{eq:semi}) and (\ref{eq:grass}) by $|\M_{g}(\F_q)|$ yields
\[
q^{-r^2}
\frac
{\prod_{i = 0}^{g-r-1}(q^g-q^i)}
{q^{(g-r)g}}
\frac
{q^{r^2}}
{|\GL_{r}(\F_q)|}
\frac
{q^{(g-r)^2}}
{|\GL_{g-r}(\F_q)|}
\frac
{|\GL_g(\F_q)|}
{q^{g^2}}
\]
which, as $g$ goes to infinity with $r$ fixed, approaches the quantity
of (\ref{eq:corankr}).

By \cite{Holland:counting}, for integers $g > s \geq r \geq 0$ the number of
  $\sigma$-endomorphisms of $V$ such that
  $\rank(M) = s$ and $\rank(M^{\infty}) = r$ is
  \[
 % \#N(g,s,r) :=
\frac{
\left(\prod_{i = 0}^{r-1}(q^{g} - q^i)\right)
|\GL_{g-r}(\F_q)|
\left(\prod_{i = 0}^{s-r-1}(q^{g-r-1} -  q^i)\right)
q^{r(g-r)}
}{
|\GL_{s-r}(\F_q)| \cdot
    |\GL_{g-s}(\F_q)| \cdot |\M_{s-r,g-s}(\F_q)}|.
\]
It follows that the probability that a random $\sigma$-endomorphism of $V$
has $\rank(M^\infty) = g-s$, conditional on   $\rank(M) = g-r$ approaches
\begin{equation}
q^{r-s}  \prod_{i=r}^{s-1} (1-q^{-i})  \prod_{i=1}^{r} (1-q^{-i})  \prod_{i=1}^{s-r} (1-q^{-i})^{-1}
\label{eq:prs}
\end{equation}
as $g$ approaches infinity.
Multiplying \eqref{eq:prs} by the constant in Proposition~\ref{pr:anumber} yields the desired result.

\end{proof}

Summing \eqref{eq:corank} with $s$ fixed and $r$ between $1$ and $s$ gives the probability that the $p$-corank of $D$ is $s$; for example, the probability that the corank is $1$ is $q^{-1} \prod_{i=1}^\infty(1+q^{-i})^{-1}$, and the probability that the corank is $2$ is $(q^{-2} + q^{-3})\prod_{i=1}^\infty(1+q^{-i})^{-1}$.

\subsection*{The group of $\F_q$-rational points of a random pqp $\DBT_1$}
\label{ss:rationalPoints}

Let $G$ be a $\BT_1$ with Dieudonn\'e module $W$. 
 The group $G(\F_q)$ of its $\F_q$-rational points is the subgroup of $G(\overline{\F}_q)$
fixed by $\Frob_q$.

\begin{prop} 
\label{P:connComp}
The group of $\F_q$-rational points of the group scheme associated to
a random pqp $\DBT_1$ has cardinality $p^d$ with probability
\begin{equation}
p^{-d^2} \prod_{i=1}^d (1-p^{-i})^{-1} \prod_{j=d+1}^\infty (1-p^{-j}).
\label{eq:cohenlenstra}
\end{equation}
\end{prop}

\begin{proof}

We again use the description of $\phi$ (from the proof of
  Proposition  \ref{p:corank}) as a random
  corank-$d$ $\sigma$-endomorphism of the $g$-dimensional vector space $FW$.
  Let $(X,Y,\phi_X,\phi_Y)$ be a quadruple where $X \oplus Y = FW$ is
  a direct sum decomposition, $\phi_X$ is a nilpotent $\sigma$-endomorphism of
  $X$ with corank $d$, and $\phi_Y$ is a $\sigma$-automorphism of $Y$.  Then
  $\phi_X \oplus \phi_Y$ is a corank-$d$ $\sigma$-endomorphism of $FW$.
  Conversely, any choice of a corank-$d$ $\sigma$-endomorphism $\phi$ yields a
  quadruple as above by taking $Y$ to be the subspace of $FW$ on which
  $\phi$ acts invertibly and $X$ the subspace on which $\phi$ acts
  nilpotently.  So a uniformly chosen corank-$d$ $\sigma$-endomorphism of $FW$
  is the same as a uniformly chosen quadruple $(X,Y,\phi_X,\phi_Y)$.
  In particular, the action of $F$ on $F^\infty W$ is precisely
  $\phi_Y$, which is drawn uniformly at random from the set of $\sigma$-automorphisms of $Y$.
  
  Now $Y$ is itself a Dieudonn\'{e} module, on which $F$ is bijective.
  Fix a positive integer $N$. Let $\DD_p$ be the category of rank-$N$ $p$-torsion Dieudonn\'{e} modules with $\F_q$ coefficients on which $F$ is bijective,  $\GG_p$ the category of rank-$N$ \'{e}tale group schemes over $\F_q$ killed by $p$, and $\FF_p$ the category of $N \times N$ matrices over $\F_p$.  The morphisms from $x$ to $y$ are, respectively:  isomorphisms of Dieudonn\'{e} modules from $x$ to $y$; group scheme isomorphisms from $x$ to $y$; changes of basis intertwining $x$ and $y$.  All three of these categories are finite groupoids and all three are equivalent.
 
The uniformity of $\phi_Y$ implies by Proposition~\ref{p:groupoidquotient} that $Y$ is uniformly distibuted in $\DD_p$.  Thus, the \'{e}tale group scheme $G_Y/\F_q$ is uniformly distributed in $\GG_p$ and the matrix $M_Y$ giving the action of $\Frob_q$ on $G_Y(\overline{\F}_q)$ is uniformly distributed in $\FF_p$.  It is precisely the dimension of $\coker(M_Y -1)$ whose distribution we are trying to study.  But applying Proposition~\ref{p:groupoidquotient} again, the distribution of $\coker(M_Y - 1)$ for $M_Y$ uniformly distributed in $\FF_p$ is identical with the distribution obtained by letting $M_Y$ be a random element of the set $\GL_N(\F_p)$.  But the distribution of $\dim \coker(M_Y - 1)$ when $M_Y$ is a random invertible matrix is  well-known to
  approach the value \eqref{eq:cohenlenstra} as $N = \dim Y \ra \infty$. 
  It is easy to see from Proposition~\ref{pr:corank} that $\dim Y$ is
  larger than any constant multiple of $g$ with probability $1$; this
  finishes the proof.
\label{TODO:30}
\end{proof}

\begin{cor}
The expected number of surjections from a random pqp $\BT_1$ to the constant group scheme $(\Z/p\Z)^d$ is $1$.
\label{co:clmoments}
\end{cor}

We note that the distribution produced here is identical with the
Cohen-Lenstra heuristic for the distribution of $p$-ranks of imaginary
quadratic fields.  This is quite natural when one considers the $p$-torsion in the Jacobian of a random
hyperelliptic curve $C$ over a finite field.  When the field has
characteristic $p$, a heuristic of the form ``random hyperelliptic
curves have random pqp $\DBT_1$" would suggest that the finite abelian
$p$-group $\Jac(C)[p](\F_q)$ is distributed according to the
Cohen-Lenstra law --- in other words, that the conjectural distribution of $\Jac(C)[p](\F_q)$ is exactly
the same whether $C$ is defined over a finite field of characteristic $p$ or of characteristic prime to $p$ (as long as the field contains no $p$th roots of unity.)
  
In the case where $C$ is defined over a finite field $k$ whose characteristic is prime to $p$, the results of
\cite{EllenbergVW:cohenLenstra2}
 prove that the $d$th moment in
Corollary~\ref{co:clmoments}
 is indeed $1$ as long as $|k|$ is sufficiently
large relative to $p$ and not congruent to 1 mod $p$.  It would be interesting to see whether there
is any way to adapt the methods of
\cite{EllenbergVW:cohenLenstra2}
 to prove that similar statements hold in characteristic $p$.

\subsection*{Other statistics and questions: final types and Newton polygons}
\label{ss:otherStatistics}

We discuss some further problems which fit into our general framework, but which we have not investigated. \\

A more refined invariant of a $\DBT_1$ of dimension $2g$ is the {\bf
  final type}, a $g$-tuple $\tau = (x_1,\ldots, x_g)$ of non-decreasing
integers such that $x_1 \in \{0,1\}$ and $x_{i+1} \leq x_{i}
+ 1$. Such a tuple determines the isomorphism class of the corresponding group scheme
over an algebraic closure of $\F_q$ and, conversely, any such tuple
arises as the final type of a $\DBT_1$; see
\cite{pries:pDivThoughts}*{2.3}. We note that, unlike $a$-number
and $p$-rank, the final type
of a $\DBT_1$ depends on $F$ and $V$, not just
their restrictions to the maximal isotropic subspaces $\ker F$ and
$\ker V$.

Another invariant of a Dieudonn\'e module $D$ is its {Newton
  Polygon} -- setting $q = p^m$, the \textbf{Newton polygon} of $D$  
has, for every root $\alpha$ (counted with multiplicity) of the characteristic polynomial of $F^m$, 
 a segment of slope $\ord_p(\alpha)/m$; when $q = p$
this is just the Newton polygon of the characteristic polynomial of $F$. The Newton
polygon determines $D$ up to isogeny; see
\cite{Manin:commutativeFormalGroups}*{II, \S 4.1}. We note that,
unlike the $a$-number, $p$-rank, or the final type,
the Newton polygon of a Dieudonn\'e module $D$ is not determined by $D/pD$.
The question of which Newton polygons are
compatible with which final types is a subject of active
research ~\cite{Oort:foliationsArticle}.\\

\begin{questions}

\begin{itemize}
\item []

\item[(1)] The $p$-rank of a Dieudonn\'e module is equal to the number of
  segments of the Newton polygon of slope zero; a natural
  generalization of Theorem \ref{p:corank} is thus the following.   For $\lambda\in \Q\cap(0,1)$, let $\mult_{\lambda}(D)$ be the
  number of segments of slope $\lambda$ in the Newton polygon of
  $D$. Does $\mult_{\lambda}(\cdot)$ converge to a distribution as $g
  \to \infty$? Moreover, can one  compute, for a fixed non-negative integer $d$, the probability that $\mult_{\lambda}(D) = d$ or the
  average value of $\mult_{\lambda}(D)$?

\item[(2)] More generally:  the Newton polygon $D$ of a random Dieudonn\'{e} module has a local-local part $D^{\loc}$ as defined in Theorem~\ref{DieudonneMain}; $D^{\loc}$ has rank $2c$, where $c$ is the $p$-corank of $D$, and the Newton polygon of $D^{\loc}$ has all slopes in the open interval $(0,1)$.  Our expectation is that the probability distribution on the Newton polygon of $D^{\loc}$, conditional on the $p$-corank of $D$ being $c$, should be given by the probability distribution on Newton polygons of nilpotent
  $p$-autodual matrices on $\Zp^{2c}$.  We expect that one can compute this distribution by force for small $c$.

\item [(3)] One can generalize either of these questions by picking,
  for each $g$, a subset $S_g$ of the set of possible Newton polygons
  (resp., final types) and asking for the proportion of Dieudonn\'e
  modules whose appropriate invariant lies in $S_g$. 
 Of course, some conditions on $S_g$ will be necessary to ensure that the proportion approaches a limit. An example where we expect a positive answer would be that in which $S_g$ is the set of final types with
$\tau_{g-1} = g-1-s$ for a fixed integer $s$.
\end{itemize}
  
\end{questions}

\section{Random curves, random abelian varieties, and random $p$-divisible groups}

\label{ss:geometry}

So far, the content of this paper has been purely combinatorial; we
have computed moments and distributions of various statistics on
random pqp $\BT_1$'s and random $p$-divisible groups.  In practice,
pqp $p$-divisible groups typically arise from motives.  In this
section, we address the question of whether $p$-divisible groups
arising from random members of a family of abelian varieties are
random $p$-divisible groups in the sense of (\ref{eq:random}) below.

Let $M_1, M_2, \ldots $ be a family of schemes (or Deligne-Mumford stacks) and let $A_i$ be an abelian scheme over $M_i$.  The three cases we will consider are:
\begin{itemize}
\item $M_g = \HH_g$, the moduli space of hyperelliptic genus $g$ curves, and $A_g$ the Jacobian of the universal curve;
\item $M_g = \MM_g$, and $A_g$ the Jacobian of the universal curve;
\item $M_g  = \AA_g$, and $A_g$ the universal abelian $g$-fold.
\end{itemize}

We say that the $p$-divisible groups associated to such a family are ``random" with respect to a statistic $X$ if
\begin{equation}
\label{eq:random}
\lim_{g \ra \infty} \frac{\sum_{y \in M_g(\F_q)} X(A_{g,y}[p^\infty])}{|M_g(\F_q)|} = \EE X.
\end{equation}

Which of these families, with respect to which statistics, yield random $p$-divisible groups?  In this section we discuss the numerical evidence concerning this question, and some geometric properties of strata of moduli spaces in characteristic $p$ which seem closely related to the statistics in the first part of the paper.

\subsection*{Relation with geometry of moduli spaces}

In this short section we record some remarks about the relationship between the heuristics presented here and the cohomology of moduli spaces of curves and abelian varieties in positive characteristic.  There are no theorems in this section, only suggestive relationships between conjectures.

The prediction that the mod $p$ Dieudonn\'{e} module of the Jacobian of a random hyperelliptic curve over $\F_q$ is a random pqp $\DBT_1$ implies, in particular, that
\begin{equation}
\lim_{g \ra \infty} \HH_g^{\no}(\F_q) / \HH_g(\F_q) \ra 1 - \prod_{i=1}^\infty(1+q^{-i})^{-1} = 1/q + 1/q^3 + 1/q^4 + \ldots
\label{eq:mg}
\end{equation}
where $\HH_g^{\no}$ denotes the non-ordinary locus, a divisor in $\HH_g$.   (One can make an analogous guess with $\MM_g$ in place of $\HH_g$.)  Thus, the heuristic goes hand in hand with a belief that the non-ordinary locus is an {\em irreducible} divisor, at least for all sufficiently large $g$ -- otherwise, the ratio would have leading term $n/q$ instead of $1/q$, where $n$ is the number of $\F_q$-rational components of $\HH_g^{\no}$.

We emphasize that almost nothing is known about the irreducibility of the non-ordinary locus in $\HH_g$ or $\MM_g$ (see \cite[\S 3.2]{AchterP:monodromypRankMg} and \cite[\S 3.7]{achterP:pRankHyperelliptic}).  The non-ordinary locus in $\AA_g$, by contrast, is known to be irreducible.

For a family of curves over $\F_q$ with random $p$-divisible groups,
Proposition~ \ref{pr:corank} shows that the proportion of curves with
$a$-number $r$ and $p$-corank $s$ has leading term $q^{-{r +1 \choose
    2} + r -s}$, which suggests that the locus of cuves with
$a$-number $r$ and $p$-corank $s$ is an irreducible locally closed
subvariety of codimension ${r +1 \choose 2} + r -s$.  This is in fact
the codimension in $\AA_g$ of the locus of abelian varieties with
$a$-number $r$ and $p$-corank $s$ (see \cite{pries:pDivThoughts}*{2.3}); so the heuristics arising from random Dieudonn\'{e} modules can be read as supportive of (or supported by) the expectation that various natural loci of curves intersect the strata in $\AA_g/\F_q$ with the expected dimension.

The heuristic \eqref{eq:mg} can also be used to make guesses about the
cohomology of various strata in $\HH_g$ and $\MM_g$.  For example:
suppose that the restriction map from the cohomology of $\MM_g$ to the
cohomology of the closed subscheme $\MM_g^{\no}$ were an isomorphism
in some large range.  Then one might expect the ratio
$\MM_g^{\no}(\F_q) / \MM_g(\F_q)$ to be very close to $1/q$, contrary
to what the heuristic predicts. {\em Proving} any implication of this kind is well out of reach -- the Betti numbers of $\MM_g$ grow superexponentially in $g$, so even with control of the cohomology groups in some large range there is no hope of using Weil bounds to get a good approximation to $\MM_g(\F_q)$~\cite{deJongK:counting}.  

For hyperelliptic curves, the situation is a bit more legible.  For simplicity, write $X_n/\F_p$ for the configuration space parametrizing monic degree-$n$ squarefree polynomials $f(x)$, and let $X_n^{\no}$ be the closed subscheme parametrizing those polynomials such that the hyperelliptic curve
\beq
y^2 = f(x)
\eeq
is non-ordinary.  It is easy to check that $|X_n(\F_q)| = q^n -
q^{n-1}$ when $n \geq 1$; moreover, the \'etale cohomology of $X_n$ is concentrated in degrees $0$ and $1$.

If \eqref{eq:mg} holds, we would have
\beq
X_n^{\no}(\F_q) \sim (q^n - q^{n-1})(1/q + 1/q^3 + 1/q^4 + \ldots) = q^{n-1} - q^{n-2} + q^{n-3} \ldots
\eeq
This suggests that $X_n^{\no}$ has cohomology beyond the classes pulled
back from $X_n$; for instance, there should be a cohomology class in
some even degree generating a subspace on which Frobenius acts with
trace $q^{n-3}$.  Moreover, this class might be expected to vanish
when $\car \F_q = 3$, since the numerical data below suggests that in
characteristic $3$ the proportion of non-ordinary hyperelliptic curves
is precisely $1/q$. 

\begin{problem} Construct such a class in the locus of non-ordinary hyperelliptic curves.
\end{problem}

Corollary~\ref{co:clmoments} says that, on average, a random $\DBT_1$
admits a single surjection to $(\Z/p\Z)^d$.  Thus, in a family of
curves $X$ with random $p$-divisible group, parametrized by a moduli
scheme $M$, the average number of surjections from $\Jac(X)(\F_q)$ to
$(\Z/p\Z)^d$ is $1$.  This suggests that the moduli space $M_{p,d}$ is
irreducible, where $M_{p,d}$ is the moduli space parametrizing curves
$X$ in $M$ together with a level structure $\Jac(X) \ra (\Z/p\Z)^d$.
And this irreducibility for every $d$ suggests that the monodromy
representation of the moduli space of ordinary curves in $M$ on the
$g$-dimensional space of \'etale $p$-torsion in $\Jac(X)$ has image
containing $\SL_g(\F_p)$.  In fact, one could refine
Corollary~\ref{co:clmoments} to apply under supplementary conditions
on $p$-corank, $a$-number, etc. ---%emily dickenson
 and the result would be a prediction that, on any of these $p$-adic strata, the monodromy in the \'etale $p$-torsion of $\Jac(X)$ has full image.  In fact, such theorems are already known in the case $M=\MM_g$~\cite{AchterP:monodromypRankMg} and $M = \HH_g$~\cite[\S 3.7]{achterP:pRankHyperelliptic}.  It seems reasonable to hope that the results of the those two papers could be used to prove that random hyperelliptic curves and random curves satisfy a weak version of the heuristic suggested by Corollary~\ref{co:clmoments}, where a limit $q \ra \infty$ is taken prior to the limit $g \ra \infty$.  This would be exactly analogous to the method used by Achter in \cite{Achter:distributionClassGroups} to derive a similarly weakened Cohen-Lenstra conjecture from a large-monodromy theorem in $\ell$-adic cohomology.

\section*{Experiments}
\label{sec:experiments}

The tables below contain experimental information about the
distribution of
$a$-numbers and orders mod $p$ of Jacobians of hyperelliptic and plane
curves.\\

The constants appearing in the tables are defined as follows. As in Theorem \ref{P:aNumber}
we define the constant $\MG(q,r)$
to be 
$$\MG(q,r) :=  q^{- {r +1 \choose 2}} \prod_{i=1}^\infty(1+q^{-i})^{-1} \prod_{i=1}^r (1-q^{-i})^{-1}.$$ 
Similarly, for a finite group $G$ of $p$-power order we define the {\bf Cohen-Lenstra probability} 
to be
$$C_p(G) := \frac{1}{|\Aut(G)|}\prod_{j=1}^\infty (1-p^{-j}),$$
and, finally, we define the {\bf truncated Malle-Garton constant} to be
$$\TMG(q;b) := \prod_{j=1}^{b} (1-q^{1-2j}).$$ 

Table \ref{Table:hyperellipticDistributionANumber} contains
distributions of $a$-numbers of Jacobians of hyperelliptic curves; we
explain below how the computations were done. As noted in the
introduction, the data suggests that the probability that the Jacobian
of a random hyperelliptic curve has $a$-number 0 does not approach the
value given by our heuristics. Rather, for $q = 3$ the data suggests
that the true probability is $2/3 = \TMG(3;1)$,  and for $q = 5$ it
suggests $0.7936 = \TMG(5;2)$. To
verify this, for $q = 5$ we took exhaustive data for low $g$ (i.e., computed the
$a$-number of the Jacobian of \emph{every} hyperelliptic curve of
genus $g$). For larger $g$ it is unreasonable to do an exhaustive
computation; for $g = 21$ we computed the $a$-numbers of the Jacobians
of 819200000 random hyperelliptic curves, and indeed the proportion
which were ordinary was closer to the truncated constant. For $q > 5$,
we did not generate enough data to distinguish between the
Malle-Garton constant and the truncated variant. 

It is natural to ask what the ``truncated" version of $\MG(q,r)$ should be for larger values of $r$. For instance, the proportion of hyperelliptic curves over $\F_3$ with $a$-number 1 appears to converging to a value $0.296....$ What limiting value (presumably a power series in $1/3$) is suggested by this experimental result?
\\

Table \ref{Table:planeDistributionANumber} contains
distributions of $a$-numbers of Jacobians of plane curves. 
The sample sizes are necessairly smaller than those of Table
\ref{Table:hyperellipticDistributionANumber} (see the comments in the
next section). 

The data for $q > 2$ agrees well with our heuristics, and in particular the truncation phenomenon
disappears (or the discrepancy from heuristics is too small for us to measure.) For $q = 2$, the data does not agree with our heuristics, and for this fact we have no conceptual explanation.  In particular, we do not see an explanation for this discrepancy along the lines of Theorem~\ref{T:oddPlane} below.
\\

Table \ref{Table:hyperellipticDistributionCL} (resp., Table
\ref{Table:planeDistributionCL}) contains the proportion of
hyperelliptic curves (resp., plane curves) $C$ such that
$p \nmid |\Jac_C(\F_q)|$ (where $p = \car \F_q$). 
For $q \neq 2$ (resp. $p > 2$) the data is consistent with the
heuristics suggested by Proposition \ref{P:connComp}. For $C$ hyperelliptic, since one can efficiently compute
the zeta function of $\Jac_C$ (and can thus detect when $p$ exactly
divides $|\Jac_C (\F_q)|$)  we also report the probability that
$\Jac_C[p^\infty](\F_q) \cong \Z/p\Z$. 
\\

\begin{rem}
  We find in Table~\ref{Table:planeDistributionCL} a notable
  divergence between experiment and heuristic for smooth plane curves
  in characteristic $2$; it appears that for plane curves $X$ of odd
  degree over $\F_q$ with $q = 2^m$, the order of $\Jac(X)(\F_q)$ is
  almost always even. This was puzzling to us until we
  realized that the behavior had a natural explanation; so natural
  that in fact we can prove that the behavior persists for plane
  curves of every odd degree.
\end{rem}

\begin{thm}
\label{T:oddPlane}
Let $d$ be a positive odd integer, let $k$ be a finite field of characteristic $2$, and let $F \in k[X,Y,Z]$ be a homogeneous degree $d$ form cutting out a smooth curve $X$ in $\P^2/k$.  Suppose furthermore that at least one monomial $X^a Y^b Z^c$ with two odd exponents has nonzero coefficient in $F$.  Then $|\Jac(X)(k)|$ is even.  In particular, the proportion of smooth degree-$d$ plane curves $X/k$ with $|\Jac(X)(k)|$ even goes to $1$ as $d$ goes to $\infty$.
\end{thm}

%  Let $M_d$ be the open subscheme of $\P^{2d+1 \choose 2}_{\F_2}$ parametrizing smooth degree-$2d+1$ plane curves $f(x,y,z) = 0$.  Then 
%\beq
%\lim_{d \ra \infty} \frac{\#\{f \in M_d(\F_q) \text{ s.t. }
%2 \nmid |\Jac_f(\F_q)|\}}
%{|M_d(\F_q)|} = 1.
%\eeq

\begin{proof}
Let $\ell_1, \ell_2$ be distinct linear forms, and let $\omega$ be the exact differential $d(\ell_1 / \ell_2)$ on $X$.  The divisor of $\omega$ is automatically a square, since locally one is just computing the derivative of a Laurent series in $k((t))$, and all such derivatives lie in the field of squares $k((t^2))$.  Let $D$ be the divisor such that $2D = \Div(\omega)$; then $D$ is evidently a half-canonical divisor, and its divisor class is independent of the choice of $\omega$ (see e.g., \cite[\S 3]{StohrV:cartier}.)

On the other hand, the canonical class on a degree-$d$ plane curve is
$(d-3)$ times the hyperplane class.  Thus, the divisor
$(1/2)(d-3)\Div(\ell_2)$ 
is also a half-canonical.  The difference between these two half-canonicals is a $2$-torsion point on $\Jac(X)(k)$; it remains to show that this point is nonzero under the given conditions.

Note that $D - (1/2)(d-3)\Div(\ell_2)$ is principal if and only if the
principal divisor $\Div(\omega) - (d-3)\Div(\ell_2)$ is the divisor of
a function $f \in (k(X)^*)^2$.
Moreover, a direct computation shows that $\Div(\omega) - (d-3)\Div(\ell_2)$
% previously $D - (1/2)(d-3)\Div(\ell_2)$ 
is the divisor of the function
\beq
\frac{dF}{d \ell_1} \ell_2^{1-d}
\eeq
which is a square only if $dF/d \ell_1$ is the square of a homogeneous
form of degree $(1/2)(d-1)$.  It is easy to see that this is
equivalent to the failure of the condition in the theorem.

\end{proof}

We remark that the converse to Theorem~\ref{T:oddPlane} does not hold;
even when the two half-canonicals constructed in the proof do agree,
there is no reason there might not be another $\F_q$-rational $2$-torsion point on $\Jac(X)$.

\subsection*{Methods of computation}

Let $C$ be a curve over $\F_q$ and let $(W, F, V, \overline{\omega})$
be the pqp $\DBT_1$ associated to the
$p$-torsion subgroup scheme of the Jacobian of $C$. Then there exists
a canonical
isomorphism $W \cong H^1_{\text{dR}}(C)$ such that the induced action
of $F$ (resp., $V$) on $H^1_{\text{dR}}(C)$ is equal to the action of
Frobenius (resp., the Cartier operator) \cite{Oda:dR}*{Section 5}, and $\overline{\omega}$
agrees with the cup product pairing. The actions of $F$ and $V$
respect the Hodge filtration
\[
\xymatrix{
0 \ar[r]& H^0(X, \Omega^1_C) \ar[r]& H^1_{\text{dR}}(C) \ar[r]& H^1(C, \mathcal{O}_C) \ar[r]& 0. 
}
\]
Moreover, the action of $F$ on $H^0(X, \Omega)$ is visibly
trivial and, dually, $V(H^1_{\text{dR}}(C)) = H^0(X, \Omega)$; in
particular, there exists a basis of $H^1_{\text{dR}}(C)$ with respect
to which the semilinear maps $F$ and $V$ correspond to the matricies
\begin{equation}
\label{eq:frob}
 \mat{0}{B}{0}{D} \text{ and }
 \mat{A}{C}{0}{0},
\end{equation}
where $A,B,C,D \in \M_g(\F_q)$. Since $\ker V = \im F$, it follows that
$a(W) = \dim (\ker F \cap \ker V)$ is equal to the nullity of $A$; by duality this is the same as the
nullity of $D$.

The matrix $A$ is called the \textbf{Cartier-Manin matrix} of $C$. When $C$ is a
hyperelliptic curve with affine equation $y^2 = f(x)$, there exists a
basis of $H^1_{\text{dR}}(C)$ with respect to which $A$ is given by
the explicit formula $(c_{pi-j})$, where $f(x)^{\frac{p-1}{2}} = \sum
c_kx^k$ and $p = \text{char } \F_q$; see \cite{Yui:jacobians} for details. In particular, one can quickly
compute the Cartier-Manin matrix, and thus $a$-number, of a hyperelliptic curve.

Moreover, one can efficiently compute $|\Jac_C(\F_{q})|$ mod $p$ (where
$q = p^r$) from the Cartier-Manin matrix. Indeed, $|\Jac_C(\F_q)| =
P(1)$, where $P$ is the characteristic polynomial of the
$r$th power of Frobenius acting
on $H^1_{\text{\'et}}(C; \Q_{\ell})$ for any $\ell \neq p$; $P$ has integral
coefficients and its reduction mod $p$ is equal to the characteristic
polynomial of the $r$th power of Frobenius acting on $H^1_{\text{dR}}(C)$. Choosing
a basis so that the matrix corresponding to $F$ is as in (\ref{eq:frob}), $F^r$ will be of
the form 
\[
F^r = \mat{0}{B'}{0}{D'},\, D' =  D\cdot\sigma(D)\cdots\sigma^{r-1}(D)
\]
and it thus suffices to
compute the value of the characteristic polynomial of $D'$ at 1. We can
therefore quickly compute the order of $\Jac_C (\F_q)$ mod $p$ from the
Cartier-Manin matrix of $C$.
\\

When $C$ is a plane curve, we do not know of
an explicit formula for $A$ or $D$ in terms of the coefficients of the defining equations
of $C$. There is however an algorithm, implemented as Magma's
\verb+CartierRepresentation+ function, which,
for a particular curve $C$ computes a representative of $A$. This
computation is much slower than in the hyperelliptic case;
accordingly, the sample sizes are smaller for plane curves.\\

Magma code which performs these computations can be found at the Arxiv
page for this paper or at
the third author's web page \cite{CaisEZB:randomDieudonneTranscript}.

\begin{center}
\begin{longtable}{|c|c|c|c|c|c|c|} 
\caption{Distribution of $a$-numbers of Jacobians of hyperelliptic curves.}\label{Table:hyperellipticDistributionANumber} \\
%%%%%%%%%%%%%%%%%%%%%%%%%%%%%%%%%%%%%%%%%%%%%%%%%%%%%%%%%%%%%%%%%%%%%%%%%%%%%%%%%%%%%%%%%%%%%%%%%%%%%%%%%%%%%%%%%%%%%%%%

\hline \multicolumn{1}{|c|}{\,\,$q$\,\,} & \multicolumn{1}{|c|}{genus} & \multicolumn{1}{|c|}{Sample Size} & \multicolumn{1}{|c|}{$a$}& \multicolumn{1}{|c|}{Proportion} & \multicolumn{1}{c|}{$\MG(q,a)$} & 
\multicolumn{1}{c|}{$\TMG(p;\frac{p-1}{2})$} \\ \hline \endfirsthead

\multicolumn{7}{c}%
{{\bfseries \tablename\ \thetable{} -- continued from previous page}} \\
\hline \multicolumn{1}{|c|}{\,\,$q$\,\,} & \multicolumn{1}{|c|}{genus} & \multicolumn{1}{|c|}{Sample Size} & \multicolumn{1}{|c|}{$a$}& \multicolumn{1}{|c|}{Proportion} & \multicolumn{1}{c|}{$\MG(q,a)$} & 
\multicolumn{1}{c|}{$\TMG(p;\frac{p-1}{2})$} \\ \hline
\endhead

\hline \multicolumn{7}{|c|}{{Continued on next page}} \\ \hline
\endfoot

\hline 
\endlastfoot
 3   &  25    &  40960000  & 0  &  0.666716  &  0.639005  &  0.666666 \\
      &         &            & 1  &  0.296272              &  0.319502  &  \\
      &         &            & 2  &  0.0328910            &  0.0399378  &  \\ \cline{2-7}
      &  100  &  5120000  & 0  &  0.666497  &  0.639005  &  0.666666 \\
      &         &            & 1  &  0.296487  &  0.319502  &  \\
      &         &            & 2  &  0.0329145  &  0.0399378  &  \\ \hline
5    &     5  & exhaustive &  0 &  0.793278 &  0.793335  &  0.793600 \\ \cline{2-7}
      &    6   & exhaustive     &  0  &  0.793875 &    &   \\ \cline{2-7}
      &   7   &   exhaustive    &  0 &   0.793557 &    &   \\ \cline{2-7}
      &  21  &  819200000 & 0  & 0.793838  &  0.793335  &  0.793600 \\ \cline{2-7}
      &  25  &  40960000  & 0  &  0.793529  &  0.793335  &  0.793600 \\
      &        &            & 1  &  0.198172  &  0.198334  &  \\
      &        &            & 2  &  0.00822029  &  0.00826392  &  \\ \cline{2-7}
      & 100 &  5120000  & 0  &  0.793838  &  0.793335  &  0.793600 \\
      &        &            & 1  &  0.197818  &  0.198334  &   \\
      &        &            & 2  &  0.00826679  &  0.00826392  &   \\ \hline
7    &   25  &  40960000  & 0  &  0.854542  &  0.854593  &  0.854594 \\
      &        &            & 1  &  0.142490  &     0.142432  &   \\
      &        &            & 2  &  0.00295969  &  0.00296733  &   \\ \hline
9    &   25 & 12735000  & 0  &  0.888970  &  0.887655  &  0.888889 \\
      &        &            & 1  &  0.109666 &  0.110957  &   \\
      &        &            & 2  &  0.00134582  &  0.00138696  &   \\ \cline{2-7}
      & 100 & 15500  & 0  &  0.888774  &  0.887655  &  0.888889 \\
      &       &             & 1  &  0.109871  &  0.110957  &   \\
      &       &             & 2  &  0.00135484  &  0.00138696  &   \\ \hline
25  &  25 &  12640000  & 0  &  0.959962  &  0.959939  &  0.959939 \\
      &       &                    & 1  &  0.0399742     &  0.0399975 &   \\
      &       &                    & 2  &  6.43987E-5  &  6.40985E-5  &   \\ \cline{2-7}
      & 100 &  1036000  & 0  &  0.959822  &  0.959939  &  0.959939 \\
      &       &            & 1  &  0.0401110  & 0.0399975    &   \\
      &       &            & 2  &  6.75676E-5  &  6.40985E-5  &   \\ \hline
27  &  25  & 13030000  & 0  & 0.962955   &  0.962914  &  0.962963 \\
      &       &                  & 1  &  0.0369945 &  0.0370352  &   \\
      &       &                  & 2  &  5.07291E-5 &  5.08724E-5  & \\ \cline{2-7}
    & 100 & 1044000  & 0  &  0.962741  &  0.962914  &  0.962963 \\
      &       &                  & 1  & 0.0372021 &  0.0370352  &   \\
      &       &                  & 2  & 5.74713E-5  &  5.08724E-5  & \\ \hline
49  &  25  &  20480000  & 0  &  0.979568  &  0.979583  &  0.979583 \\ \hline
81  &  25  &  20480000  & 0  &  0.987662  &  0.987653  &  0.987655 \\ \hline
125 &  25 &  20480000  & 0  &  0.991984  &  0.991999  &  0.991999 \\ \hline

%%%%%%%%%%%%%%%%%%%%%%%%%%%%%%%%%%%%%%%%%%%%%%%%%%%%%%%%%%%%%%%%%%%%%%%%%%%%%%%%%%%%%%%%%%%%%%%%%%%%%%%%%%%%%%%%%%%%%%%%
\end{longtable}
\end{center}
%%%%%%%%%%%%%%%%%%%%%%%%%%%%%%%%%%%%%%%%%%%%%%%%%%%%%%%%%%%%%%%%%%%%%%%%%%%%%%%%%%%%%%%%%%%%%%%%%%%%%%%%%%%%%%%%%%%%%%%%

%%%%%%%%%%%%%%%%%%%%%%%%%%%%%%%%%%%%%%%%%%%%%%%%%%%%%%%%%%%%%%%%%%%%%%%%%%%%%%%%%%%%%%%%%%%%%%%%%%%%%%%%%%%%%%%%%%%%%%%%
\begin{center}
\begin{longtable}{|c|c|c|c|c|c|c|} 
\caption{Distribution of $a$-numbers of Jacobians of plane curves.}\label{Table:planeDistributionANumber} \\
%%%%%%%%%%%%%%%%%%%%%%%%%%%%%%%%%%%%%%%%%%%%%%%%%%%%%%%%%%%%%%%%%%%%%%%%%%%%%%%%%%%%%%%%%%%%%%%%%%%%%%%%%%%%%%%%%%%%%%%%

\hline \multicolumn{1}{|c|}{\,\,$q$\,\,} &
\multicolumn{1}{|c|}{degree} &\multicolumn{1}{|c|}{genus} & \multicolumn{1}{|c|}{$a$} & \multicolumn{1}{|c|}{Sample Size} & \multicolumn{1}{|c|}{Proportion} & \multicolumn{1}{c|}{$\MG(q,a)$} \\ \hline \endfirsthead

\multicolumn{7}{c}%
{{\bfseries \tablename\ \thetable{} -- continued from previous page}} \\
\hline \multicolumn{1}{|c|}{\,\,$q$\,\,} &
\multicolumn{1}{|c|}{degree}   &
\multicolumn{1}{|c|}{genus} & \multicolumn{1}{|c|}{$a$} & \multicolumn{1}{|c|}{Sample Size} & \multicolumn{1}{|c|}{Proportion} & \multicolumn{1}{c|}{$\MG(q,a)$} \\ \hline
\endhead

\hline \multicolumn{7}{|c|}{{Continued on next page}} \\ \hline
\endfoot

\hline 
\endlastfoot
2  &  7  &  15    & 0 &  56949615  &  0.426022  &  0.419422 \\    
    &     &         & 1 &        &  0.422294  &  0.419422 \\
    &     &         & 2 &                &      0.109071  &  0.139807 \\ \cline{2-7}
    &  10  &  36  & 0 & 80000  &  0.423363  &  0.419423      \\ \hline
3  &  7  &  15   & 0 & 3062000  &  0.638947  &  0.639006    \\
    &     &         & 1 &    &  0.319267   & 0.319502  \\
    &     &         & 2 &                &  0.0404950  & 0.0399378  \\\cline{2-7}
    &  8  &  21   & 0 & 249230  &  0.638133  &  0.639006 \\ \cline{2-7}
    &  10  &  36  & 0 & 154000  & 0.639273  &  0.639006 \\ 
    &     &         & 1 &                & 0.318792 & 0.319502  \\
    &     &         & 2 &                & 0.0404481 & 0.0399378 \\ \hline
4  &  7  &  15    & 0 & 2782000  & 0.737809  &  0.737513 \\    
    &     &           & 1 &               &  0.245737  &  0.245837 \\
    &      &           & 2 &              &  0.0152703  &  0.0163892 \\  \cline{2-7}
    &  10  &  15    & 0 & 521000  &  0.739500 &   \\    
    &     &           & 1 &               &  0.242875  & \\
    &      &           & 2 &              &  0.0174167  &   \\ \hline
5  &  7  &  15   & 0 & 590000  &  0.793784  &  0.793335 \\\cline{2-7}
    &  10  &  36  & 0 & 17851  &  0.796874  &  0.793335 \\ \hline
9  &  7  &  15   & 0 & 1080000 &  0.887926  &  0.887654  \\
    &     &         & 1 &               &  0.110636  & 0.110957  \\
    &     &         & 2 &               &   0.00143426  & 0.00138696 \\ \cline{2-7}
    &  10  &  36   & 0 & 102000  & 0.888040     &    \\
    &     &         & 1 &                & 0.110549  &  \\
    &     &         & 2 &                & 0.00140196 & \\ \hline 
25&  7  &  15   & 0 & 563000   & 0.960135      &   0.959938 \\
    &     &         & 1 &               & 0.0398114   &   0.0399975 \\
    &     &         & 2 &               &  5.33808E-5 &  6.40985E-5 \\ \cline{2-7}
    & 10 &  36   & 0 & 36000    & 0.958667     &    \\
    &     &         & 1 &                & 0.0412778  &  \\
    &     &         & 2 &                & 5.55555E-5  & \\ \hline 
27&  7  &  15   & 0 & 947000   &  0.962757   &   0.962914 \\
    &     &         & 1 &               &  0.0371953  &  0.0370352 \\
    &     &         & 2 &               &   4.75185E-5  & 5.08724E-5  \\ \cline{2-7}
    & 10 &  36   & 0 & 89000    & 0.962023     &    \\
    &     &         & 1 &                & 0.0378652  &  \\
    &     &         & 2 &                & 0.000112360  & \\ \hline

%27 10

%%%%%%%%%%%%%%%%%%%%%%%%%%%%%%%%%%%%%%%%%%%%%%%%%%%%%%%%%%%%%%%%%%%%%%%%%%%%%%%%%%%%%%%%%%%%%%%%%%%%%%%%%%%%%%%%%%%%%%%%
\end{longtable}
\end{center}
%%%%%%%%%%%%%%%%%%%%%%%%%%%%%%%%%%%%%%%%%%%%%%%%%%%%%%%%%%%%%%%%%%%%%%%%%%%%%%%%%%%%%%%%%%%%%%%%%%%%%%%%%%%%%%%%%%%%%%%%

%%%%%%%%%%%%%%%%%%%%%%%%%%%%%%%%%%%%%%%%%%%%%%%%%%%%%%%%%%%%%%%%%%%%%%%%%%%%%%%%%%%%%%%%%%%%%%%%%%%%%%%%%%%%%%%%%%%%%%%%
\begin{center}
\begin{longtable}{|c|c|c|c|c|c|} 
\caption{Distribution of $|\Jac[p^\infty](\F_q)|$ for Jacobians of
  hyperelliptic curves.}\label{Table:hyperellipticDistributionCL} \\
%%%%%%%%%%%%%%%%%%%%%%%%%%%%%%%%%%%%%%%%%%%%%%%%%%%%%%%%%%%%%%%%%%%%%%%%%%%%%%%%%%%%%%%%%%%%%%%%%%%%%%%%%%%%%%%%%%%%%%%%

\hline \multicolumn{1}{|c|}{\,\,$q$\,\,} & \multicolumn{1}{|c|}{genus}
&\multicolumn{1}{|c|}{$|\Jac(\F_q)[p^\infty]|$} & \multicolumn{1}{|c|}{Sample Size} &
\multicolumn{1}{|c|}{Proportion} & \multicolumn{1}{c|}{$C_p(G)$} \\
\hline \endfirsthead

\multicolumn{6}{c}%
{{\bfseries \tablename\ \thetable{} -- continued from previous page}} \\
\hline \multicolumn{1}{|c|}{\,\,$q$\,\,} & \multicolumn{1}{|c|}{genus}  & \multicolumn{1}{|c|}{$|\Jac(\F_q)|$} & \multicolumn{1}{|c|}{Sample Size} & \multicolumn{1}{|c|}{Proportion} & \multicolumn{1}{c|}{$C_p(r)$} \\ \hline
\endhead

\hline \multicolumn{6}{|c|}{{Continued on next page}} \\ \hline
\endfoot

\hline 
\endlastfoot
3  &  12  &  $p$  &  20000  &  0.280100  &  0.280063 \\ \cline{2-6}
   &  20  &  1  &  1600000  &  0.560527  &  0.560126 \\ \cline{2-6}
   &  25  &  1  &  400000  &  0.560198  &  0.560126 \\ \hline
5  &  8  &  $p$  &  5000  &  0.193200  &  0.190083 \\ \cline{2-6}
   &  15  &  1  &  1600000  &  0.759874  &  0.760333 \\ \cline{2-6}
   &  25  &  1  &  400000  &  0.760077  &  0.760333 \\ \hline
7  &  15  &  1  &  1600000  &  0.837408  &  0.836796 \\ \cline{2-6}
   &  25  &  1  &  400000  &  0.836867  &  0.836796 \\ \hline
9   &  12  &  1  & 274016000  & 0.560169   & 0.560126  \\ \hline
11 &  15  &  1  &  1600000  &  0.900908  &  0.900833 \\ \cline{2-6}
   &  25  &  1  &  400000  &  0.900988  &  0.900833 \\ \hline
25   &  12  &  1  & 269210000 & 0.760389   &   0.760333 \\ \hline
 
%%%%%%%%%%%%%%%%%%%%%%%%%%%%%%%%%%%%%%%%%%%%%%%%%%%%%%%%%%%%%%%%%%%%%%%%%%%%%%%%%%%%%%%%%%%%%%%%%%%%%%%%%%%%%%%%%%%%%%%%
\end{longtable}
\end{center}
%%%%%%%%%%%%%%%%%%%%%%%%%%%%%%%%%%%%%%%%%%%%%%%%%%%%%%%%%%%%%%%%%%%%%%%%%%%%%%%%%%%%%%%%%%%%%%%%%%%%%%%%%%%%%%%%%%%%%%%%

%%%%%%%%%%%%%%%%%%%%%%%%%%%%%%%%%%%%%%%%%%%%%%%%%%%%%%%%%%%%%%%%%%%%%%%%%%%%%%%%%%%%%%%%%%%%%%%%%%%%%%%%%%%%%%%%%%%%%%%%
\begin{center}
\begin{longtable}{|c|c|c|c|c|c|} 
\caption{Probability that $ p \nmid |\Jac(\F_q)|$ for Jacobians of plane curves.}\label{Table:planeDistributionCL} \\
%\label{table:planepnotdivjac}
%%%%%%%%%%%%%%%%%%%%%%%%%%%%%%%%%%%%%%%%%%%%%%%%%%%%%%%%%%%%%%%%%%%%%%%%%%%%%%%%%%%%%%%%%%%%%%%%%%%%%%%%%%%%%%%%%%%%%%%%

\hline \multicolumn{1}{|c|}{\,\,$q$\,\,} & \multicolumn{1}{|c|}{degree} &\multicolumn{1}{|c|}{genus} & \multicolumn{1}{|c|}{Sample Size} & \multicolumn{1}{|c|}{Proportion} & \multicolumn{1}{c|}{$C_p(1)$} \\ \hline \endfirsthead

\multicolumn{6}{c}%
{{\bfseries \tablename\ \thetable{} -- continued from previous page}} \\
\hline \multicolumn{1}{|c|}{\,\,$q$\,\,} & \multicolumn{1}{|c|}{degree}  & \multicolumn{1}{|c|}{genus} & \multicolumn{1}{|c|}{Sample Size} & \multicolumn{1}{|c|}{Proportion} & \multicolumn{1}{c|}{$C_p(1)$} \\ \hline
\endhead

\hline \multicolumn{6}{|c|}{{Continued on next page}} \\ \hline
\endfoot

\hline 
\endlastfoot
 2 &  7   & 15 & 37940000  & 0.071762  & 0.288788 \\ 
   &  8   & 21 & 1777100    & 0.229230  & \\ %0.288788 \\ 
    &  9   & 28 & 313000      & 0.000223  &  \\ 
   &  10 & 36 & 349500 & 0.246071 &  \\ %\hline
    &  11 & 55 & 312000      & 0.0000064 &  \\ \hline
3 &  7   & 15 & 137000      & 0.553302   & 0.560126 \\ 
   &  10 & 36 & 142000 & 0.559676 &   \\ \hline
4 & 9   & 28 & 47775  & 0.000000  & 0.288788   \\
   & 10  & 36 & 31200  & 0.285673 & \\ \hline
5 &  7   & 15 & 350000 & 0.760462 & 0.760332   \\ 
   &  10 & 36 & 48000 & 0.761500 &  \\ \hline
8 &  7   & 15 &  219725  & 0.000000  &   0.288788 \\ 
   &  8   & 21 & 89575   & 0.283349  &  \\ \hline
9 &  10   & 36 & 57000 & 0.558772  & 0.560126   \\ 
\hline 

%%%%%%%%%%%%%%%%%%%%%%%%%%%%%%%%%%%%%%%%%%%%%%%%%%%%%%%%%%%%%%%%%%%%%%%%%%%%%%%%%%%%%%%%%%%%%%%%%%%%%%%%%%%%%%%%%%%%%%%%
\end{longtable}
\end{center}
%%%%%%%%%%%%%%%%%%%%%%%%%%%%%%%%%%%%%%%%%%%%%%%%%%%%%%%%%%%%%%%%%%%%%%%%%%%%%%%%%%%%%%%%%%%%%%%%%%%%%%%%%%%%%%%%%%%%%%%%

 \def\cprime{$'$}

% \bibliographystyle{hep}
%  \bibliography{/Users/davidMBrownJr/Documents/Dropbox/academic/latex/jabref/master.bib}

\end{document}